\documentclass[11pt]{amsart} 
\usepackage{amsfonts} 
\usepackage{graphicx}
\usepackage{color}

\newtheorem{theorem}{Theorem}[section]
\newtheorem{proposition}[theorem]{Proposition}
\newtheorem{lemma}[theorem]{Lemma}
\newtheorem{corollary}[theorem]{Corollary}
\newtheorem{definition}[theorem]{Definition}

\newtheorem{conjecture}[theorem]{Conjecture}

\theoremstyle{definition}
\newtheorem{construction}[theorem]{Construction}
\newtheorem{remark}[theorem]{Remark}

\renewenvironment{proof}{{\noindent\bf Proof.}}{\hfill $\Box$\par\vskip3mm}

\renewcommand{\qedsymbol}{$\Box$}
\DeclareMathOperator{\tb}{tb}
\DeclareMathOperator{\rot}{r}

\begin{document}

\def\reals{{\mathbb R}}
 \def\ch{{\mathcal H}}
 \def\cA{{\mathcal A}}
 \def\cD{{\mathcal D}}
 \def\cK{{\mathcal K}}
 \def\cC{{\mathcal C}}
 \def\cN{{\mathcal N}}
 \def\cR{{\mathcal R}}
 \def\cS{{\mathcal S}}
 \def\cT{{\mathcal T}}
 \def\cV{{\mathcal V}}
 \def\ta{{\mathcal T}_{\subset}}
 \def\cI{{\mathcal I}}
 \def\bC{{\bf C}}
 \def\axis{{\bf A}}
 \def\fibr{{\bf H}}
 \def\ba{{\bf a}}
 \def\bb{{\bf b}}
 \def\bc{{\bf c}}
 \def\be{{\bf e}}
 \def\d{{\delta}} 
 \def\ci{{\circ}} 
 \def\e{{\epsilon}} 
 \def\l{{\lambda}} 
 \def\L{{\Lambda}} 
 \def\m{{\mu}} 
 \def\n{{\nu}} 
 \def\o{{\omega}} 
 \def\s{{\sigma}} 
 \def\v{{\varphi}} 
 \def\a{{\alpha}} 
 \def\b{{\beta}} 
 \def\p{{\partial}} 
 \def\r{{\rho}} 
 \def\ra{{\rightarrow}} 
 \def\lra{{\longrightarrow}} 
 \def\g{{\gamma}} 
 \def\D{{\Delta}} 
 \def\La{{\Leftarrow}} 
 \def\Ra{{\Rightarrow}} 
 \def\x{{\xi}} 
 \def\c{{\mathbb C}} 
 \def\z{{\mathbb Z}} 
 \def\2{{\mathbb Z_2}} 
 \def\q{{\mathbb Q}} 
 \def\t{{\tau}} 
 \def\u{{\upsilon}} 
 \def\th{{\theta}} 
 \def\la{{\leftarrow}} 
 \def\lla{{\longleftarrow}} 
 \def\da{{\downarrow}} 
 \def\ua{{\uparrow}} 
 \def\nwa{{\nwtarrow}} 
 \def\swa{{\swarrow}} 
 \def\nea{{\netarrow}} 
 \def\sea{{\searrow}} 
 \def\hla{{\hookleftarrow}} 
 \def\hra{{\hookrightarrow}} 
 \def\sl{{SL(2,\mathbb C)}} 
 \def\ps{{PSL(2,\mathbb C)}} 
 \def\qed{{\hfill$\diamondsuit$}} 
 \def\pf{{\noindent{\bf Proof.\hspace{2mm}}}} 
 \def\ni{{\noindent}} 
 \def\sm{{{\mbox{\tiny M}}}} 
 \def\sc{{{\mbox{\tiny C}}}}

\title[Studying uniform thickness II]{Studying uniform thickness II:\\Transversely non-simple iterated torus knots}

\author{Douglas J. LaFountain}
\address{Centre for Quantum Geometry of Moduli Spaces \\ Aarhus University}
\email{dlafount@imf.au.dk}
\urladdr{http://pure.au.dk/portal/en/dlafount@imf.au.dk}


\begin{abstract}
We prove that an iterated torus knot type fails the uniform thickness property (UTP) if and only if all of its iterations are positive cablings, which is precisely when an iterated torus knot type supports the standard contact structure.  We also show that all iterated torus knots that fail the UTP support cabling knot types that are transversely non-simple. 
\end{abstract}

\maketitle

\section{Introduction}

Let $K$ be a knot type in $S^3$ with the standard tight contact structure $\xi_{std}$.  The {\em uniform thickness property} (UTP) is fundamental to understanding embeddings of solid tori representing $K$ in $(S^3, \xi_{std})$; in brief, $K$ satisfies the UTP if every such solid torus thickens to one with convex boundary slope $1/\overline{\tb}(K)$.  If there exists a solid torus that does not exhibit thickening, $K$ fails the UTP, and such a solid torus is said to be {\em non-thickenable}.  The UTP was first introduced by Etnyre and Honda \cite{[EH1]}, who showed that the $(2,3)$-torus knot fails the UTP by identifying such non-thickenable tori.  They then used this to show that the $(2,3)$-torus knot supports a transversely non-simple cabling, i.e., a knot type obtained by taking a $(p,q)$ curve on the boundary of a tubular neighborhood of $K$.  In joint work with Etnyre and Tosun \cite{[ELT]}, we extended this study to show that all positive $(p,q)$-torus knots fail the UTP, and support non-simple cablings; furthermore, we established a complete Legendrian and transverse classification for cables of positive torus knots through the study of {\em partially thickenable tori}. In \cite{[L]}, we also showed that the general class of knot types $K$ which both satisfy the UTP and are Legendrian simple is closed under the operation of cabling.  An application of this was the identification of large classes of Legendrian simple iterated torus knot types, i.e., iterated cablings of torus knots.



In this paper we determine precisely which iterated torus knot types satisfy the UTP, and which fail the UTP; this is the first complete UTP classification for a large class of knots.  We also prove that failure of the UTP for an iterated torus knot type is a sufficient condition for the existence of transversely non-simple cablings of that knot.  Specifically, we have the following two theorems and corollary:

\begin{theorem}
\label{main theorem}
Let $K_r = ((P_1,q_1),...,(P_i,q_i),...,(P_r,q_r))$ be an iterated torus knot, where the $P_i$'s are measured in the standard Seifert framing, and $q_i > 1$ for all $i$.  Then $K_r$ fails the UTP if and only if $P_i > 0$ for all $i$, where $1 \leq i \leq r$.
\end{theorem}

\noindent In the second theorem, $\chi(K)$ is the Euler characteristic of a minimal genus Seifert surface for a knot $K$:

\begin{theorem}
\label{second theorem}
If $K_r$ is an iterated torus knot that fails the UTP, then it supports infinitely many transversely non-simple cablings $K_{r+1}$ of the form $(-\chi(K_r),k+1)$, where $k$ ranges over an infinite subset of positive integers.
\end{theorem}

To state our corollary to Theorem \ref{main theorem}, recall that if $K$ is a fibered knot, then there is an associated open book decomposition of $S^3$ that supports a contact structure, denoted $\xi_K$ (see \cite{[E],[TW]}).  Iterated torus knots are fibered knots, and Hedden has shown that the subclass of iterated torus knots where each iteration is a positive cabling, i.e. $P_i > 0$ for all $i$, is precisely the subclass of iterated torus knots where $\xi_{K_r}$ is isotopic to $\xi_{std}$ \cite{[He1]}.  We thus obtain the following corollary:

\begin{corollary}
An iterated torus knot $K_r$ fails the UTP if and only if $\xi_{K_r} \cong \xi_{std}$.
\end{corollary}

We make a few remarks about these theorems.  First, it will be shown that these transversely non-simple cablings all have two Legendrian isotopy classes at the same rotation number and maximal Thurston-Bennequin number $\overline{\tb}$, and thus they exhibit Legendrian non-simplicity at $\overline{\tb}$.  Second, in the class of iterated torus knots there are certainly more transversely non-simple cablings than those in Theorem \ref{second theorem}, as seen in \cite{[EH1],[ELT]}.  However, we present just the class of non-simple cablings in Theorem \ref{second theorem}, and leave a more complete classification as an open question.   

We now present a conjectural generalization of the above two theorems and corollary.  To this end, recall that Hedden has shown that for general fibered knots $K$ in $S^3$, $\xi_K \cong \xi_{std}$ precisely when $K$ is a fibered strongly quasipositive knot \cite{[He3]}; he also shows that for these knots, the maximal self-linking number is $\overline{sl}(K)=-\chi(K)$ \cite{[He2]}.  Furthermore, from the work of Etnyre and Van Horn-Morris \cite{[EV]}, we know that for fibered knots $K$ in $S^3$ that support the standard contact structure there is a unique transverse isotopy class at $\overline{sl}$.  In the present paper, all of these ideas are brought to bear on the class of iterated torus knots, and this motivates the following conjecture concerning general fibered knots:

\begin{conjecture}
Let $K$ be a fibered knot in $S^3$; then $K$ fails the UTP if and only if $\xi_K \cong \xi_{std}$, and hence if and only if $K$ is fibered strongly quasipositive.  Moreover, if a topologically non-trivial fibered knot $K$ fails the UTP, then it supports cablings that are transversely non-simple.
\end{conjecture}

Our main tools will be convex surface theory and the classification of tight contact structures on solid tori and thickened tori.  Most of the results we use can be found in \cite{[EH1], [ELT], [H1], [H2], [L]}, and if we use a result from one of these works, it will be specifically referenced.  Moreover, subsections 2.2 through 2.4 in \cite{[ELT]} provide a nice summary of much of the needed background.

The plan of the note is as follows.  In \S \ref{sec:Defns} we recall definitions, notation, and identities used in \cite{[EH1],[L]}.  In \S \ref{sec:strategy} we outline a strategy of proof of Theorem \ref{main theorem} that yields the statement of two key lemmas.  In \S \ref{sec:positivecablings} we prove the first lemma, and in \S \ref{sec:negativecablings} we prove the second lemma and complete the proof of Theorem \ref{main theorem}.  In \S \ref{sec:nonsimple} we prove Theorem \ref{second theorem}.

\bigskip

\noindent\textbf{Acknowledgements.}  We would like to thank William Menasco, John Etnyre, and B$\ddot{\textrm{u}}$lent Tosun for both their insight and interest.  This work was partially supported by QGM (Centre for Quantum Geometry of Moduli Spaces) funded by the Danish National Research Foundation.

\section{Definitions, notation, and identities}
\label{sec:Defns}

\subsection{Iterated torus knots}

{\em Iterated torus knots}, as topological knot types, can be defined recursively.  Let 1-iterated torus knots be simply torus knots $(p_1,q_1)$ with $p_1$ and $q_1$ co-prime nonzero integers, and $|p_1|, q_1 > 1$.  Here, as usual, $p_1$ is the algebraic intersection with a longitude, and $q_1$ is the algebraic intersection with a meridian in the preferred Seifert framing for a torus representing the unknot.  Then for each $(p_1,q_1)$ torus knot, take a tubular neighborhood $N((p_1,q_1))$; the boundary of this is a torus, and given a framing we can describe simple closed curves on that torus as co-prime pairs $(p_2,q_2)$, with $q_2 > 1$.  In this way we obtain all 2-iterated torus knots, which we represent as ordered pairs, $((p_1,q_1),(p_2,q_2))$.  Recursively, suppose the $(r-1)$-iterated torus knots are defined; we can then take tubular neighborhoods of all of these, choose a framing, and form the $r$-iterated torus knots as ordered $r$-tuples $((p_1,q_1),...,(p_{r-1},q_{r-1}),(p_r,q_r))$, again with $p_r$ and $q_r$ co-prime, and $q_r > 1$.

For ease of notation, if we are looking at a general $r$-iterated torus knot type, we will refer to it as $K_r$; a Legendrian representative will usually be written as $L_r$.

We will study iterated torus knots using two framings.  The first is the standard Seifert framing for a torus, where the meridian bounds a disc inside the solid torus, and we use the preferred longitude which bounds a surface in the complement of the solid torus.  We will refer to this framing as $\mathcal{C}$.  The second framing is a non-standard framing using a different longitude that comes from the cabling torus.  More precisely, to identify this non-standard longitude on $\partial N(K_r)$, we first look at $K_r$ as it is embedded in $\partial N(K_{r-1})$.  We take a small neighborhood $N(K_r)$ such that $\partial N(K_r)$ intersects $\partial N(K_{r-1})$ in two parallel simple closed curves.  These curves are longitudes on $\partial N(K_r)$ in this second framing, which we will refer to as $\mathcal{C'}$.  Note that this $\mathcal{C'}$ framing is well-defined for any cabled knot type.  Moreover, for purpose of calculations there is an easy way to change between the two framings, which will be reviewed below.

Given a simple closed curve $(\mu, \lambda)$ on a torus, measured in some framing as having $\mu$ meridians and $\lambda$ longitudes, we will say this curve has slope of $\frac{\lambda}{\mu}$; i.e., longitudes over meridians.  Therefore we will refer to the longitude in the $\mathcal{C'}$ framing as $\infty^\prime$, and the longitude in the $\mathcal{C}$ framing as $\infty$.  The meridian in both framings will have slope $0$.  These are the conventions used in \cite{[EH2],[ELT],[L]}.

We will also use a convention that meridians in the standard $\mathcal{C}$ framing, that is, algebraic intersection with $\infty$, will be denoted by upper-case $P$'s.  On the other hand, meridians in the non-standard $\mathcal{C}'$ framing, that is, algebraic intersection with $\infty'$, will be denoted by lower-case $p$'s.  These are the conventions used in \cite{[L]}.  Given a curve $L = (P,q)$ on a torus $\partial N$, there is then a relationship between the framings $\mathcal{C}'$ and $\mathcal{C}$ on $\partial N(L)$.  In terms of a change of basis, we can represent slopes $\lambda/\mu$ as column vectors and then get from a slope $\lambda/\mu'$, measured in $\mathcal{C}'$ on $\partial N(L)$, to a slope $\lambda/\mu$, measured in $\mathcal{C}$, by:

\begin{equation}
\label{changeofbasis}
\displaystyle \left(\begin{array}{cc}
														1 & Pq\\
														0 & 1\end{array}\right)\left(\begin{array}{c}
														\mu '\\
														\lambda\end{array}\right) = \left(\begin{array}{c}
														\mu\\
														\lambda\end{array}\right)\nonumber
										\end{equation}

In other words, $\mu = \mu ' + Pq\lambda$.

Given an iterated torus knot type $K_r = ((p_1,q_1),...,(p_r,q_r))$ where the $p_i$'s are measured in the $\mathcal{C}'$ framing, we define two quantities.  The two quantities are:

\begin{equation}
\label{ArBr}
\displaystyle A_r := \sum_{\alpha=1}^{r}p_\alpha \prod_{\beta=\alpha+1}^{r}q_\beta \prod_{\beta=\alpha}^{r}q_\beta \ \ \ \ \ \ \ \ B_r := \sum_{\alpha=1}^r \left(p_\alpha \prod_{\beta=\alpha+1}^r q_\beta \right) + \prod_{\alpha=1}^r q_\alpha
\end{equation}

Note here we use a convention that $\prod_{\beta=r+1}^{r}q_\beta := 1$.  Also, if we restrict to the first $i$ iterations, that is, to $K_i = ((p_1,q_1),...,(p_i,q_i))$, we have an associated $A_i$ and $B_i$.  For example, 

\begin{equation}
A_i := \sum_{\alpha=1}^{i}p_\alpha \prod_{\beta=\alpha+1}^{i}q_\beta \prod_{\beta=\alpha}^{i}q_\beta\nonumber
\end{equation}

From Section 3 in \cite{[L]} we obtain four useful identities which we will apply extensively throughout this note:

\begin{equation}
\label{identities}
A_r = q_r^2 A_{r-1} + p_rq_r \qquad  B_r = q_r B_{r-1} + p_r \qquad  P_r=q_rA_{r-1} + p_r \qquad  A_r = P_r q_r
\end{equation}

We conclude with a computation of the Euler characteristic for iterated torus knots obtained through positive cablings (see also Lemma 3.3 in \cite{[L]}).

\begin{lemma}
\label{eulerchar}
Suppose $K_r = ((P_1,q_1),...,(P_r,q_r))$ is an iterated torus knot where $P_i > 0$ for all $i$.  Then $-\chi(K_r)=A_r-B_r$.
\end{lemma}

\begin{proof}
A formula for $\chi(K_r)$ is given at the end of the proof of Corollary 3 in \cite{[BW]}.  In the notation used in that paper, the formula is $\chi(K_r)=\prod_{i=1}^r p_i - \sum_{i=1}^r q_i(p_i-1)\prod_{j=i+1}^r p_j$, since in our case all the $e_i=1$ as we are cabling positively at each iteration.  However, note that our $(P_i,q_i)$ corresponds to $(q_i,p_i)$ in \cite{[BW]} for $i > 1$.  We thus obtain the equation

\begin{equation}
\displaystyle \chi(K_r) = P_1 \prod_{i=2}^r q_i - q_1(P_1 - 1)\prod_{i=2}^r q_i - \sum_{i=2}^r P_i(q_i-1)\prod_{j=i+1}^r q_j \nonumber
\end{equation} 

Examination of this formula for $\chi(K_r)$ yields the following recursive expression using our $P$'s and $q$'s: 

\begin{eqnarray}
\chi(K_r) &=& q_r \left[P_1 \prod_{i=2}^{r-1} q_i - q_1(P_1 - 1)\prod_{i=2}^{r-1} q_i - \sum_{i=2}^{r-1} P_i(q_i-1)\prod_{j=i+1}^{r-1} q_j \right] - P_r(q_r-1)\nonumber\\
 &=& q_r \chi(K_{r-1}) - P_r q_r + P_r\nonumber
\end{eqnarray}

Now for a positive torus knot $(P_1,q_1)$, we have $\chi=-A_1+B_1$, so we can inductively assume the lemma holds for $K_{r-1}$.  Thus using the recursive expression we have

\begin{eqnarray*}
\chi(K_r) &=& q_r \chi(K_{r-1}) - P_r q_r + P_r\\\nonumber
					&=& q_r(-A_{r-1}+B_{r-1}) - A_r + q_rA_{r-1} + p_{r-1}\\\nonumber
					&=&-A_r + B_r\nonumber
\end{eqnarray*}
\end{proof}

\subsection{Legendrian knots, convex tori, and the UTP}

Recall that for Legendrian knots embedded in $S^3$ with the standard tight contact structure, there are two classical invariants of Legendrian isotopy classes, namely the Thurston-Bennequin number, $\tb$, and the rotation number, $\rot$.  For a given topological knot type, if the ordered pair $(\rot,\tb)$ completely determines the Legendrian isotopy classes, then that knot type is said to be {\em Legendrian simple}.  For transverse knots there is one classical invariant, the self-linking number $sl$; for a given topological knot type, if the value of $sl$ completely determines the transverse isotopy classes, then that knot type is said to be {\em transversely simple}.  For a given topological knot type, if we plot Legendrian isotopy classes at points $(\rot,\tb)$, we obtain a plot of points that takes the form of a {\em Legendrian mountain range} for that knot type.

We will be examining Legendrian knots which are embedded in convex tori.  Recall that the characteristic foliation induced by the contact structure on a convex torus can be assumed to have a standard form, where there are $2n$ parallel {\em Legendrian divides} and a one-parameter family of {\em Legendrian rulings}.  Parallel push-offs of the Legendrian divides gives a family of $2n$ {\em dividing curves}, referred to as $\Gamma$.  For a particular convex torus, the slope of components of $\Gamma$ is fixed and is called the {\em boundary slope} of any solid torus which it bounds; however, the Legendrian rulings can take on any slope other than that of the dividing curves by Giroux's Flexibility Theorem \cite{[G]}.  A {\em standard neighborhood} of a Legendrian knot $L$ will have two dividing curves and a boundary slope of $1/\tb(L)$.

We can now state the definition of the {\em uniform thickness property} as given by Etnyre and Honda \cite{[EH1]}.  For a knot type $K$, define the {\em contact width} of $K$ to be

\begin{equation}
w(K)=\textrm{sup}\frac{1}{\textrm{slope}(\Gamma_{\partial N})}
\end{equation}

In this equation the $N$ are solid tori having representatives of $K$ as their cores; slopes are measured using the Seifert framing where the longitude has slope $\infty$; the supremum is taken over all solid tori $N$ representing $K$ where $\partial N$ is convex.  Any knot type $K$ satisfies the inequality $\overline{\tb}(K) \leq w(K) \leq \overline{\tb}(K)+1$.  A knot type $K$ then {\em satisfies the uniform thickness property (UTP)} if the following hold:

\begin{itemize}
\item[1.] $\overline{\tb}(K)=w(K)$, where $\overline{\tb}$ is the maximal Thurston-Bennequin number for $K$.
\item[2.] Every solid torus $N$ representing $K$ can be thickened to a standard neighborhood of a maximal $\overline{\tb}$ Legendrian knot.
\end{itemize} 

A solid torus $N$ {\em fails to thicken} if for all $N' \supset N$, we have $\textrm{slope}(\Gamma_{\partial N'}) = \textrm{slope}(\Gamma_{\partial N})$.  Thus one of the ways a knot type $K$ may fail the UTP is if it is represented by a solid torus $N$ which fails to thicken, and such that $\textrm{slope}(\Gamma_{\partial N}) \neq 1/\overline{\tb}(K)$.


Given a Legendrian curve $L = (P,q)$ on a convex torus $\partial N$, we define $t$ to be the twisting of the contact planes along $L$ with respect to the $\mathcal{C}'$ framing on $\partial N(L)$; in this case, equation 2.1 in \cite{[EH1]} gives us:  
														
\begin{equation}
\label{5}
\tb(L) = Pq + t(L)
\end{equation}

Observe that $t(L)$ is also the twisting of the contact planes with respect to the framing given by $\partial N$, and so is equal to $-1/2$ times the geometric intersection number of $L$ with $\Gamma_{\partial N}$.  $\overline{t}$ will denote the maximal twisting number with respect to this framing.

We also had two definitions introduced in \cite{[L]} that will be useful in this note.

\begin{definition}
{\em Let $N$ be a solid torus with convex boundary in standard form, and with $\textrm{slope}(\Gamma_{\partial N})=a/b$ in some framing.  If $|2b|$ is the geometric intersection of the dividing set $\Gamma$ with a longitude ruling in that framing, then we will call $a/b$ the {\em intersection boundary slope}}.
\end{definition}

Note that when we have an intersection boundary slope $a/b$, then $2\textrm{gcd}(a,|b|)$ is the number of dividing curves.

\begin{definition}
{\em For $r \geq 1$ and positive integer $k$, define $N_r^k$ to be any solid torus representing $K_r$ with intersection boundary slope of $-(k+1)/(A_rk+B_r)$, as measured in the $\mathcal{C}'$ framing.  Also define the integer $n_r^k :=\textrm{gcd}((k+1),(A_rk+B_r))$.}
\end{definition}

Note that $N_r^k$ has $2n_r^k$ dividing curves.  Note also that the above definition is only for $k \geq 1$.  However, we will also define $N_r^0$ to be a standard neighborhood of a $\overline{\tb}(K_r)$ representative, and thus have this as the $k=0$ case.  

\begin{remark}
\label{eulercharremark}
We will be particularly interested when $K_r$ is an iterated torus knot obtained from positive cablings; in this case, note that after doing a change of coordinates from the $\mathcal{C}'$ framing to the $\mathcal{C}$ framing, one obtains that the intersection boundary slope of $N_r^k$ is $(k+1)/(A_r-B_r)$, or in other words, by Lemma \ref{eulerchar}, $-(k+1)/\chi(K_r)$.  Thus $\Gamma_{\partial N_r^k}$ intersects the Seifert longitude exactly $2(-\chi(K_r))$ times, regardless of what $k$ is; this will be vital for our arguments. 
\end{remark}

Finally, recall that if $\mathcal{A}$ is a convex annulus with Legendrian boundary components, then dividing curves are arcs with endpoints on either one or both of the boundary components.  Dividing curves that are boundary parallel are called {\em bypasses}; an annulus with no bypasses is said to be {\em standard convex}.

\subsection{Twist number lemma and the Farey tessellation}  The following lemma, due to Honda \cite{[H1]}, will play a role in this work.

\begin{lemma}[Twist number lemma, Honda]
\label{twistnumber}
Let $L$ be a Legendrian knot with twisting $n$.  Let $r$ be the slope of a Legendrian ruling curve on $\partial N(L)$.  If there exists a bypass attached along this ruling curve, and $1/r \geq (n+1)$, then passing through the bypass yields a Legendrian curve, with larger twisting, which is isotopic (but not Legendrian isotopic) to $L$.
\end{lemma}

This lemma can be thought of as a corollary to the following proposition, also due to Honda \cite{[H1]}, which describes how slopes of dividing curves change due to bypasses attached to convex tori.  Recall that fractional slopes can be placed on the boundary of the Poincar$\acute{\textrm{e}}$ disk $\mathbb{D}$ using the {\em Farey tessellation}, where two slopes with intersection number one are connected by an arc in the Farey tessellation -- see subsection 2.2.3 in \cite{[ELT]} for a complete discussion.  In the following proposition, the torus $T$ can be thought of as inheriting an orientation from the solid torus which it bounds.

\begin{proposition}[Honda]
\label{changingslopes}
Let $T$ be a convex torus in standard form with $|\Gamma_{T}|=2,$
dividing slope $s$ and ruling slope $r\not=s.$ Let $D$ be a bypass for
$T$ attached to the front of $T$ along a ruling curve. Let $T'$ be the
torus obtained from $T$ by attaching the bypass $D.$ Then
$|\Gamma_{T'}|=2$ and the dividing slope $s'$ of $\Gamma_{T'}$ is
determined as follows: let $[r,s]$ be the arc on $\partial\mathbb{D}$
running from $r$ counterclockwise to $s,$ then $s'$ is the point in
$[r,s]$ closest to $r$ with an edge to $s.$

If the bypass is attached to the back of $T$ then the same algorithm
works except one uses the interval $[s,r]$ on
$\partial\mathbb{D}$.
\end{proposition}

Thus, note that when {\em thickening} a solid torus $N$, boundary slopes change in a {\em clockwise} manner on $\partial\mathbb{D}$; and when {\em thinning} $N$, slopes change in a {\em counterclockwise} manner.  Also, note that the boundary slope of $0$ cannot be realized when the contact structure is tight.  However, given a tight solid torus $N$ with boundary slope $s$, and given $s'$ a rational slope somewhere in the interval $(s,0)$ obtained by going counterclockwise from $s$ to $0$, then there exists a solid torus $N' \subset N$ with boundary slope $s'$ (see \cite{[H1]}).

\subsection{Imbalance Principle}
As we see that bypasses are useful in changing dividing curves on a surface, we mention a standard way to try to find them called the Imbalance Principle. Suppose that $T$ and $T'$ are two disjoint convex tori and $\cA$ is a convex annulus whose interior is disjoint from $T$ and $T'$, but whose boundary is Legendrian with one component on each surface. If $|\Gamma_T \cap \partial \cA|>|\Gamma_{T'} \cap \partial \cA|$ then there will be a bypass on $\cA$ along the $T$-edge.

\subsection{Universally tight contact structures}  Recall that a contact structure $\xi$ on a 3-manifold $M$ is said to be {\em overtwisted} if there exists an overtwisted disc, and a contact structure is {\em tight} if it is not overtwisted.  Moreover, one can further analyze tight contact 3-manifolds $(M,\xi)$ by looking at what happens to $\xi$ when pulled back to the universal cover $\widetilde{M}$ via the covering map $\pi:\widetilde{M} \rightarrow M$.  In particular, if the pullback of $\xi$ remains tight, then $(M,\xi)$ is said to be {\em universally tight}.

The classification of universally tight contact structures on solid tori is known from the work of Honda.  Specifically, from Proposition 5.1 in \cite{[H1]}, we know there are exactly two universally tight contact structures on $S^1 \times D^2$ with boundary torus having two dividing curves and slope $s < -1$ in some framing.  These are such that a convex meridional disc has boundary-parallel dividing curves that separate half-discs all of the same sign, and thus the two contact structures differ by $-id$.  (If $s=-1$, there is only one tight contact structure, and it is universally tight.)

Also from the work of Honda, we know that if $\xi$ is a contact structure which is everywhere transverse to the fibers of a circle bundle $M$ over a closed oriented surface $\Sigma$, then $\xi$ is universally tight.  This is the content of Lemma 3.9 in \cite{[H2]}, and such a transverse contact structure is said to be {\em horizontal}.

\subsection{Transverse push-offs of Legendrian knots}  Given a Legendrian knot $L$, recall that there are well-defined {\em positive and negative transverse push-offs}, denoted by $T_+(L)$ and $T_-(L)$, respectively.  Moreover, the self-linking numbers of these transverse push-offs are given by the formula $sl(T_{\pm}(L))=\tb(L)\mp \rot(L)$.

\section{Strategy of proof for Theorem \ref{main theorem}}
\label{sec:strategy}

In this section we present a strategy of proof for Theorem \ref{main theorem}.  We begin with a theorem that in previous works has in effect been proved, but not stated.  In this theorem $K$ is a knot type and $K_{(P,q)}$ is the $(P,q)$-cabling of $K$.

\begin{theorem}[Etnyre-Honda, L.]
\label{UTP preserved}
If $K$ satisfies the UTP, then $K_{(P,q)}$ also satisfies the UTP.
\end{theorem}

\begin{proof}
The case where the cabling fraction $P/q < w(K)$ is the content of Theorem 1.3 in \cite{[EH1]}.  For the case where $P/q > w(K)$, the proof follows from examining the proofs of Theorem 3.2 \cite{[EH1]} and Theorem 1.1 in \cite{[L]} and observing that Legendrian simplicity of $K$ is not needed to preserve the UTP.
\end{proof}

An immediate application for our purposes is that if an iterated torus knot $K_r = ((P_1,q_1),\cdots,(P_r,q_r))$ satisfies the UTP, then $K_{r+1} = ((P_1,q_1),\cdots,(P_r,q_r),(P_{r+1},q_{r+1}))$ also satisfies the UTP.

With this theorem in mind, we will prove Theorem \ref{main theorem} by way of three lemmas, two of which combine in an induction argument.  For this purpose we make the following inductive hypothesis, which from here on we will refer to as {\em the inductive hypothesis}.\\

\noindent\textbf{Inductive hypothesis:}  Let $K_r = ((P_1,q_1),...,(P_r,q_r))$ be an iterated torus knot, as measured in the standard $\mathcal{C}$ framing.  The inductive hypothesis assumes that the following hold:

\begin{itemize}
	\item[1.] $P_i > 0$ for all $i$, where $1 \leq i \leq r$.  (Thus $A_i=P_iq_i > 0$ for all $i$ as well.)
	\item[2.] $0<\overline{\tb}(K_r)=w(K_r) \leq A_r$.  (Thus $-A_r < \overline{t}(K_r) \leq 0$.)
	\item[3.] Any solid torus $N_r$ representing $K_r$ thickens to some $N_r^k$ (including $N_r^0$ which is a standard neighborhood of a $\overline{\tb}$ representative).
	\item[4.] If $N_r$ fails to thicken then it is an $N_r^k$, and it has at least $2n_r^k$ dividing curves. 
	\item[5.] The candidate non-thickenable $N_r^k$ exist and actually fail to thicken for $k \geq C_r$, where $C_r$ is some positive integer that varies according to the knot type $K_r$.  Moreover, these $N_r^k$ that fail to thicken have contact structures that are universally tight, with convex meridian discs $D$ containing bypasses all of the same sign; i.e., the rotation number of meridian curves is $\rot(\partial D) = \pm k$.  Also, a Legendrian ruling preferred longitude on these $\partial N_r^k$ has rotation number zero for $k > 0$.\\
\end{itemize}

Another way of stating item 4 is that every solid torus $N_r$ is contained in some $N_r^k$, and if $N_r$ fails to thicken, then boundary slopes do not change in passing to the $N_r^k \supset N_r$, although the number of dividing curves may decrease.  Also, note that, by item 5, any $K_r$ which satisfies the inductive hypothesis fails the UTP.

We first observe that the inductive hypothesis is true for the base case of positive torus knots, as established in \cite{[ELT],[L]}.

\begin{lemma}
\label{basecasepositivetorusknots}
The inductive hypothesis is true for positive torus knots $K_1 = (P,q)$.
\end{lemma}

\begin{proof}
Clearly item 1 of the inductive hypothesis holds.  From \cite{[EH2]} we know that $0 < \overline{\tb}(K_1) = Pq-P-q < A_1 = Pq$; this proves part of item 2.  

The remaining part of 2 follows from Lemma 4.5 in \cite{[L]}, and items 3 and 4 hold from Lemma 4.3 in \cite{[L]} (see also Lemma 3.1 in \cite{[ELT]}).  We briefly recall the sketch of the proof of that lemma below, as we will be using similar ideas shortly in the induction step. 

The idea in Lemma 4.3 in \cite{[L]} was the following:  given a solid torus $N_1$ representing the positive torus knot $K_1$, take a neighborhood of a Legendrian Hopf link $N(L_1) \sqcup N(L_2)$ in its complement.  Then, in the complement of $N_1 \cup N(L_1) \cup N(L_2)$, join a $(P,q)$-curve on $\partial N(L_1)$ to a $(q,P)$-curve on $\partial N(L_2)$ with a standard convex annulus $\cA$ having no bypasses (this could be achieved after possibly destabilizing $L_1 \sqcup L_2$).  One could then calculate the intersection boundary slope of $-\partial (N(L_1) \cup N(L_2) \cup N(\cA))$ to be identical to one of the $N_1^k$.  This established item 3.  Then, in that same lemma, item 4 was shown by observing that if $N_1$ had the same boundary slope as an $N_1^k$, but with less than $2n_1^k$ dividing curves, then $N_1$ would in fact thicken.

Construction 3.2, and Lemmas 3.3 and 3.4 in \cite{[ELT]} then combine to establish item 5, using $C_1 = 1$.  Again, we include the ideas in those results below, as we will use similar arguments shortly in the induction step.

The idea in Construction 3.2 in \cite{[ELT]} was to take one of the universally tight $N_1^k$, with convex meridian discs having bypasses all of the same sign, and build $S^3$ with the tight contact structure around it.  Specifically, we joined two $\infty'$-longitudes on $\partial N_{1}^{k}$ by a standard convex annulus $\cA$, so that if we then let $R = N_{1}^{k} \cup N(\cA)$, we had that $R$ was diffeomorphic to $T^2 \times [0,1]$, with a $[0,1]$-invariant contact structure on $N(\cA)$.  The contact structure on $R$ could then be isotoped to be transverse to the fibers of $R$, hence a horizontal contact structure, and therefore universally tight.  With appropriate choice of dividing curves on $\cA$, we could then assure that the two toric boundaries of $R$ represented those of standard neighborhoods of our desired Legendrian Hopf link, and gluing in such neighborhoods gave us $S^3$ with the tight contact structure.  This showed that the $N_1^k$ exist.

The idea in Lemma 3.3 in \cite{[ELT]} was to show that the $N_1^k$ are non-thickenable by examining the complement $M_1^k = S^3 \setminus N_1^k$.  Specifically, since the positive torus knot $(P,q)$ was a fibered knot (with fiber $\Sigma$) with periodic monodromy, $M_1^k$ had a $Pq$-fold cover $\widetilde{M}_1^k \cong S^1 \times \Sigma$.  We then showed that the $S^1$ fibers in $\widetilde{M}_1^k$ could all be made Legendrian of the same (negative) twisting $-(A_1k+B_1)$.  We then assumed, for contradiction, that $N_1^k$ thickened, and showed this resulted in a new Legendrian, topologically isotopic to the $S^1$ fibers, with twisting $-t' > -(A_1k+B_1)$.  We then showed, after cutting $\Sigma$ into a polygon $P$ to obtain a solid torus $S^1 \times P$, that we could tile enough copies of $S^1 \times P$ together to enclose the Legendrian with twisting $-t'$ inside a standard neighborhood of a Legendrian with twisting $-(A_1k+B_1)$.  This was a contradiction, and showed that the $N_1^k$ failed to thicken.

Finally, Lemma 3.4 in \cite{[ELT]} computed rotation numbers. 
\end{proof}

Our second key lemma used in proving Theorem \ref{main theorem} is the following induction step, which, along with the base case of positive torus knots, will show that if the iterated torus knot $K_r = ((P_1,q_1), ..., (P_r,q_r))$ is such that $P_i > 0$ for all $i$, then $K_r$ fails the UTP.

\begin{lemma}
\label{indhyppreserved}
Suppose $K_r$ satisfies the inductive hypothesis, and $K_{r+1}$ is a cabling where $P_{r+1} > 0$; then $K_{r+1}$ satisfies the inductive hypothesis.
\end{lemma}

The main idea in the argument used to prove this lemma will be that since $K_r$ satisfies the inductive hypothesis, there is an infinite collection of non-thickenable solid tori whose boundary slopes form an increasing sequence converging to $-1/A_r$ in the $\mathcal{C}'$ framing (which is $\infty$ in the $\mathcal{C}$ framing).  As a consequence, it will be shown that cabling slopes with $P_{r+1} > 0$ in the $\mathcal{C}$ framing will have a similar sequence of non-thickenable solid tori.  

Our third key lemma is the following, which along with Theorem \ref{UTP preserved} and the fact that negative torus knots satisfy the UTP (see \cite{[EH1]}), will show that if at least one of the $P_i < 0$, then $K_r$ satisfies the UTP.

\begin{lemma}
\label{negativeUTP}
Suppose $K_r$ satisfies the inductive hypothesis, and $K_{r+1}$ is a cabling where $P_{r+1} < 0$; then $K_{r+1}$ satisfies the UTP.
\end{lemma} 

In the following \S \ref{sec:positivecablings} we prove Lemma \ref{indhyppreserved}, and in \S \ref{sec:negativecablings} we prove Lemma \ref{negativeUTP}.

\section{Positive cablings that fail the UTP}
\label{sec:positivecablings}

Now that we know that the base case holds for positive torus knots, we begin to prove Lemma \ref{indhyppreserved} -- for the whole of this section we will thus have that $P_{r+1} > 0$, $K_r$ satisfies the inductive hypothesis, and we work to show that $K_{r+1}$ satisfies the inductive hypothesis.  We will need to break the proof of Lemma \ref{indhyppreserved} into two cases, Case I being where $P_{r+1}/q_{r+1} > w(K_r)$, and Case II being where $w(K_r) > P_{r+1}/q_{r+1} > 0$.  However, we first note the following.  

\begin{lemma}
\label{thickeningorder}
Let $K_r$ be an iterated torus knot with $P_i > 0$ for all $i$. If $0 \leq k_1 < k_2$, then 
\begin{equation}
-\frac{k_1+1}{A_rk_1+B_r} < -\frac{k_2+1}{A_rk_2+B_r}\nonumber
\end{equation} 
\end{lemma}

\begin{proof}
Following Lemma \ref{eulerchar} and Remark \ref{eulercharremark}, in the standard $\mathcal{C}$ framing we have that $(k_1+1)/(A_r-B_r) < (k_2+1)/(A_r-B_r)$; changing coordinates to the $\mathcal{C}'$ framing yields $-(k_1+1)/(A_rk_1+B_r) < -(k_2+1)/(A_rk_2+B_r)$.
\end{proof}

We now directly address the two different cases in two different subsections.

\medskip

\subsection{Case I: $P_{r+1}/q_{r+1} > w(K_r)$.}  We work through proving items 2-5 in the inductive hypothesis via a series of lemmas.  The following lemma begins to address item 2.

\begin{lemma}
\label{overlinetbcalc}
If $P_{r+1}/q_{r+1} > w(K_r)$, then $\overline{\tb}(K_{r+1})=A_{r+1}-(P_{r+1}-q_{r+1}w(K_r)) > 0$.
\end{lemma}

\begin{proof}
The proof is similar to that of Lemma 3.3 in \cite{[EH1]} (note that our $A_{r+1} = P_{r+1}q_{r+1}$).  We first claim that $\overline{t}(K_{r+1}) < 0$.  If not, there exists a Legendrian $L_{r+1}$ with $t(L_{r+1})=0$ and a solid torus $N_r$ with $L_{r+1}$ as a Legendrian divide.  But then we would have a boundary slope of $P_{r+1}/q_{r+1} > w(K_r)$ in the $\mathcal{C}$ framing, which cannot occur.

So since $\overline{t}(K_{r+1}) < 0$, any Legendrian $L_{r+1}$ must be a ruling on a convex $\partial N_r$ with slope $0>s \geq 1/\overline{t}(K_r)$ in the $\mathcal{C}'$ framing.  But then if $s = - \lambda/\mu > 1/\overline{t}(K_r)$, we have that $t(L_r) = - (p_{r+1}\lambda + q_{r+1}\mu) < -\lambda(p_{r+1} - \overline{t}(K_r)q_{r+1}) \leq -(p_{r+1} - \overline{t}(K_r)q_{r+1})$.  This shows that $\overline{\tb}(K_{r+1})$ is achieved by a Legendrian ruling on a convex torus having slope $1/w(K_r)$ in the standard $\mathcal{C}$ framing.

Finally, note that $A_{r+1}-(P_{r+1}-q_{r+1}w(K_r)) = A_{r+1}-(q_{r+1}(A_r-w(K_r))+p_{r+1}) > A_{r+1}-(q_{r+1}^2A_r +p_{r+1}q_{r+1})=0$. 
\end{proof}

With the following lemma we prove that items 3 and 4 of the inductive hypothesis hold for $K_{r+1}$.

\begin{lemma}
\label{nonthickening inductive step}
If $P_{r+1}/q_{r+1} > w(K_r)$, let $N_{r+1}$ be a solid torus representing $K_{r+1}$, for $r \geq 1$.  Then $N_{r+1}$ can be thickened to an $N_{r+1}^{k'}$ for some nonnegative integer $k'$.  Moreover, if $N_{r+1}$ fails to thicken, then it has the same boundary slope as some $N_{r+1}^{k'}$, as well as at least $2n_{r+1}^{k'}$ dividing curves.
\end{lemma}

\begin{proof}
In this case, for the $\mathcal{C}'$ framing, we have either $p_{r+1} > 0$ or $q_{r+1}/p_{r+1} < 1/\overline{t}(K_r)$ (the latter being relevant only if $\overline{t}(K_r) < 0$); in other words, $q_{r+1}/p_{r+1}$ is clockwise from $1/\overline{t}(K_r)$ in the Farey tessellation.  The proof in this case is nearly identical to the proof of Lemma 4.4 in \cite{[L]}; we will include the details, however, as certain particular calculations differ.  Moreover, we will use modifications of this argument in Case II and thus will be able to refer to the details here.

Let $N_{r+1}$ be a solid torus representing $K_{r+1}$.  Let $L_r$ be a Legendrian representative of $K_r$ in $S^3 \backslash N_{r+1}$ and such that we can join $\partial N(L_{r})$ to $\partial N_{r+1}$ by a convex annulus $\mathcal{A}_{(p_{r+1},q_{r+1})}$ whose boundaries are $(p_{r+1},q_{r+1})$ and $\infty'$ rulings on $\partial N(L_{r})$ and $\partial N_{r+1}$, respectively.  Then topologically isotop $L_{r}$ in the complement of $N_{r+1}$ so that it maximizes $\tb$ over all such isotopies; this will induce an ambient topological isotopy of $\mathcal{A}_{(p_{r+1},q_{r+1})}$, where we still can assume $\mathcal{A}_{(p_{r+1},q_{r+1})}$ is convex.  A picture is shown in (a) in Figure \ref{fig:non-thickening1B}.  In the $\mathcal{C}'$ framing we will have $\textrm{slope}(\Gamma_{\partial N(L_{r})})=-1/m$ where $m \geq 0$, since $\overline{t}(K_{r})\leq 0$.  Now if $m=\overline{t}(K_{r})$, then there will be no bypasses on the $\partial N(L_{r})$-edge of $\mathcal{A}_{(p_{r+1},q_{r+1})}$, since the $(p_{r+1},q_{r+1})$ ruling would be at maximal twisting.  On the other hand, if $m < \overline{t}(K_{r})$, then there will still be no bypasses on the $\partial N(L_{r})$-edge of $\mathcal{A}_{(p_{r+1},q_{r+1})}$, since such a bypass would induce a destabilization of $L_{r}$, thus increasing its $\tb$ by one --  here we are using the twist number lemma, Lemma \ref{twistnumber} above.  To satisfy the conditions of this lemma, we are using the fact that either $p_{r+1} > 0$ or $q_{r+1}/p_{r+1} < 1/\overline{t}(K_r)$.  Furthermore, we can thicken $N_{r+1}$ through any bypasses on the $\partial N_{r+1}$-edge, and thus assume $\mathcal{A}_{(p_{r+1},q_{r+1})}$ is standard convex.

Now let $N_{r} := N_{r+1} \cup N(\mathcal{A}_{(p_{r+1},q_{r+1})}) \cup N(L_r)$.  Inductively we can thicken $N_r$ to an $N_r^k$ with intersection boundary slope $-(k+1)/(A_rk+B_r)$ where $k$ is minimized over all such thickenings (if we have $k=0$, then we will have $N_{r+1}$ thickening to a standard neighborhood of a knot at $\overline{\tb}$ -- see the proof of Theorem 1.1 in Section 2 in \cite{[L]}; so we can assume $k > 0$).  Then consider a convex annulus $\widetilde{\mathcal{A}}$ from $\partial N(L_{r})$ to $\partial N_{r}^k$, such that $\mathcal{\widetilde{A}}$ is in the complement of $N_{r}$ and $\partial \mathcal{\widetilde{A}}$ consists of $(p_{r+1},q_{r+1})$ rulings.  A picture is shown in (b) in Figure \ref{fig:non-thickening1B}.  By an argument identical to that used in Lemma 4.4 in \cite{[L]}, $\mathcal{\widetilde{A}}$ is standard convex; we briefly recall the details below for completeness. 

\begin{figure}[htbp]
\label{nonthickening1}
	\centering
		\includegraphics[width=.90\textwidth]{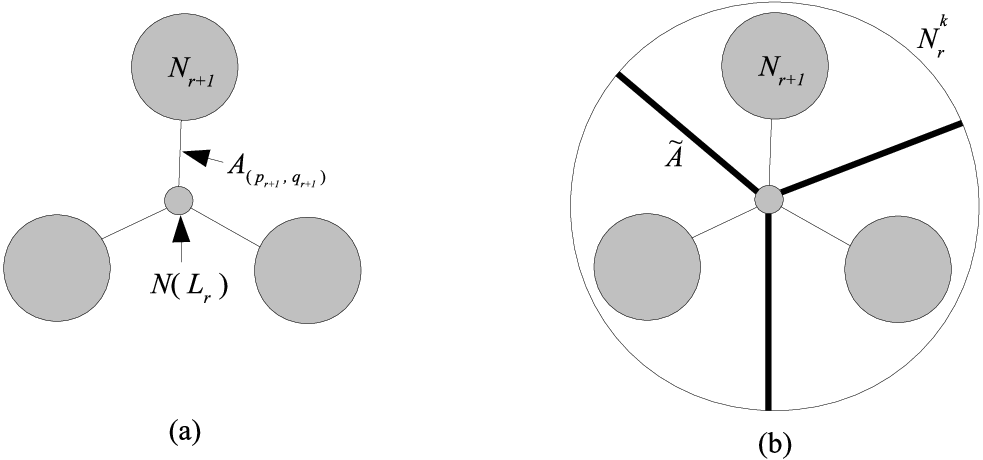}
	\caption{{\small $N_{r+1}$ is the larger solid torus in gray; $N(L_{r})$ is the smaller solid torus in gray.}}
	\label{fig:non-thickening1B}
\end{figure}

Certainly there are no bypasses on the $\partial N(L_{r})$-edge of $\mathcal{\widetilde{A}}$; furthermore, any bypasses on the $\partial N_r^k$-edge must pair up via dividing curves on $\partial N_r^k$ and cancel each other out as in part (a) of Figure \ref{nonthickening2}, for otherwise a bypass on $\partial N(L_r)$ would be induced via the annulus $\mathcal{\widetilde{A}}$ as in part (b) of Figure \ref{nonthickening2}.  As a consequence, allowing $N_r^k$ to thin inward through such bypasses does not change the boundary slope, but just reduces the number of dividing curves to less than $2n_r^k$.  But then inductively we can thicken this new $N_r^k$ to a smaller $k$-value, contradicting the minimality of $k$.  Thus $\mathcal{\widetilde{A}}$ is standard convex.

\begin{figure}[htbp]
	\centering
		\includegraphics[width=.80\textwidth]{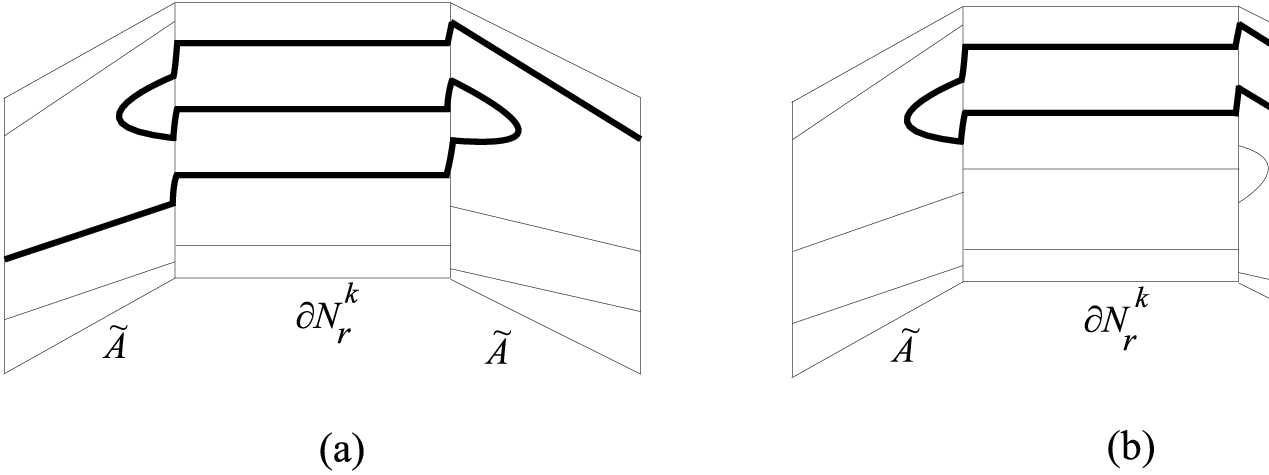}
	\caption{{\small Part (a) shows bypasses that cancel each other out after edge-rounding.  Part (b) shows a bypass induced on $\partial N(L_r)$ via $\widetilde{A}$.}}
	\label{nonthickening2}
\end{figure}

Now four annuli compose the boundary of a solid torus $\widetilde{N}_{r+1}$ containing $N_{r+1}$:  the two sides of a thickened $\mathcal{\widetilde{A}}$; $\partial N_{r}^k \backslash \partial \mathcal{\widetilde{A}}$; and $\partial N(L_{r}) \backslash \partial \mathcal{\widetilde{A}}$.  We can compute the intersection boundary slope of this solid torus.  To this end, recall that $\textrm{slope}(\Gamma_{\partial N(L_{r})})=-1/m$ where $m > 0$ ($m=0$ would be the $\overline{t}$ case which we have taken care of above).  To determine $m$ we note that the geometric intersection of $(p_{r+1},q_{r+1})$ with $\Gamma$ on $\partial N_{r}^k$ and $\partial N(L_{r})$ must be equal, yielding the equality

\begin{equation}
\label{positive}
p_{r+1} + mq_{r+1} = p_{r+1}k + p_{r+1} + q_{r+1}(A_{r}k + B_{r})
\end{equation}

These equal quantities are greater than zero, since $q_{r+1}/p_{r+1}$ is clockwise from $-1/m$ (and $-(k+1)/(A_rk+B_r)$) in the Farey tessellation -- we note here that this will yield $(A_{r+1}k'+B_{r+1}) > 0$ for the calculations below.  In the meantime, however, the above equation gives

\begin{equation}
m=p_{r+1}\frac{k}{q_{r+1}}+A_{r}k + B_{r}
\end{equation}

We define the integer $k':=k/q_{r+1}$.  We now choose $(p'_{r+1},q'_{r+1})$ to be a curve on these two tori such that $p_{r+1}q'_{r+1}-p'_{r+1}q_{r+1}=1$, and we change coordinates to a framing $\mathcal{C}''$ via the map $((p_{r+1},q_{r+1}),(p'_{r+1},q'_{r+1})) \mapsto ((0,1),(-1,0))$.  Under this map we obtain 
\begin{equation}
\textrm{slope}(\Gamma_{\partial N_{r}^k}) = \frac{q'_{r+1}(A_{r}k+B_{r})+p'_{r+1}(q_{r+1}k'+1)}{A_{r+1}k'+B_{r+1}}
\end{equation}

\begin{equation}
\textrm{slope}(\Gamma_{\partial N(L_{r})})=\frac{q'_{r+1}(p_{r+1}k' +A_{r}k+B_{r}) +p'_{r+1}}{A_{r+1}k'+B_{r+1}}
\end{equation}

We then obtain in the $\mathcal{C}'$ framing, after edge-rounding, that the intersection boundary slope of $\widetilde{N}_{r+1}$ is

\begin{eqnarray}
\label{equationarray}
\textrm{slope}(\Gamma_{\partial \widetilde{N}_{r+1}})&=&\frac{q'_{r+1}(A_{r}k+B_{r})+p'_{r+1}(q_{r+1}k'+1)}{A_{r+1}k'+B_{r+1}}\nonumber\\
																			&-&\frac{q'_{r+1}(p_{r+1}k' +A_{r}k+B_{r}) +p'_{r+1}}{A_{r+1}k'+B_{r+1}}\nonumber\\
																			&-&\frac{1}{A_{r+1}k'+B_{r+1}}\nonumber\\
																			&=&-\frac{k'+1}{A_{r+1}k'+B_{r+1}}
\end{eqnarray}

We remark here that in these particular edge-rounding calculations we are using the fact that the $q_{r+1}/p_{r+1}$ rulings on both $\partial N(L_{r})$ and $\partial N_{r}^k$ intersects the dividing curves {\em positively}, which is equivalent to saying that $q_{r+1}/p_{r+1}$ is clockwise from $-1/m$ and $-(k+1)/(A_rk+B_r)$ in the Farey tessellation -- this yields the $-1/(A_{r+1}k'+B_{r+1})$ summand in the calculation above.  This will be important to remember in Case II below.

However, sticking to the current case, this shows that any $N_{r+1}$ representing $K_{r+1}$ can be thickened to one of the $N_{r+1}^{k'}$, and if $N_{r+1}$ fails to thicken, then it has the same boundary slope as some $N_{r+1}^{k'}$.  We now show that if $N_{r+1}$ fails to thicken, and if it has the minimum number of dividing curves over all such $N_{r+1}$ which fail to thicken and have the same boundary slope as $N_{r+1}^{k'}$, then $N_{r+1}$ is actually an $N_{r+1}^{k'}$. 

To see this, as above we can choose a Legendrian $L_{r}$ that maximizes $\tb$ in the complement of such a non-thickenable $N_{r+1}$, and such that we can join $\partial N(L_{r})$ to $\partial N_{r+1}$ by a convex annulus $\mathcal{A}_{(p_{r+1},q_{r+1})}$ whose boundaries are $(p_{r+1},q_{r+1})$ and $\infty'$ rulings on $\partial N(L_{r})$ and $\partial N_{r+1}$, respectively.  Again we have no bypasses on the $\partial N(L_{r})$-edge, and in this case we have no bypasses on the $\partial N_{r+1}$-edge since $N_r$ fails to thicken and is at minimum number of dividing curves.

As above, let $N_{r}:=N_{r+1} \cup N(\mathcal{A}_{(p_{r+1},q_{r+1})}) \cup N(L_{r})$.  We claim this $N_{r}$ fails to thicken.  To see this, take a convex annulus $\mathcal{\widetilde{A}}$ from $\partial N(L_{r})$ to $\partial N_{r}$, such that $\mathcal{\widetilde{A}}$ is in the complement of $N_{r+1}$ and $\partial \mathcal{\widetilde{A}}$ consists of $(p_{r+1},q_{r+1})$ rulings.  We know $\mathcal{\widetilde{A}}$ is standard convex since the twisting is the same on both edges and there are no bypasses on the $\partial N(L_{r})$-edge.  A picture is shown in Figure \ref{fig:non-thickening3BB}.

\begin{figure}[htbp]
	\centering
		\includegraphics[width=0.45\textwidth]{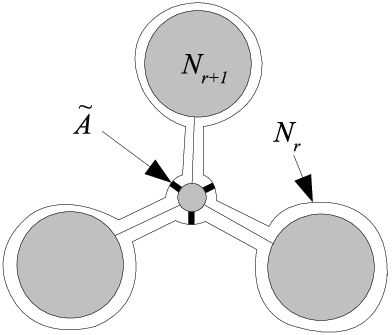}
	\caption{\small{Shown is a meridional cross-section of $N_{r}$.  The larger gray solid torus represents $N_{r+1}$; the smaller gray solid torus is $N(L_{r})$.}}
	\label{fig:non-thickening3BB}
\end{figure}

Now four annuli compose the boundary of a solid torus containing $N_{r+1}$:  the two sides of the thickened $\mathcal{\widetilde{A}}$, which we will call $\mathcal{\widetilde{A}}_+$ and $\mathcal{\widetilde{A}}_-$; $\partial N_{r} \backslash \partial \mathcal{\widetilde{A}}$, which we will call $\mathcal{A}_{r}$; and $\partial N(L_{r}) \backslash \partial \mathcal{\widetilde{A}}$, which we will call $\mathcal{A}_{L_{r}}$.  Any thickening of $N_{r}$ will induce a thickening of $N_{r+1}$ to $\widetilde{N}_{r+1}$ via these four annuli.

Suppose, for contradiction, that $N_{r}$ thickens outward so that $\textrm{slope}(\Gamma_{\partial N_{r}})$ changes.  Note that during the thickening, $\mathcal{A}_ {L_{r}}$ stays fixed.  We examine the rest of the annuli by breaking into two cases.

\textbf{Case 1:}  After thickening, suppose $\mathcal{\widetilde{A}}$ is still standard convex; that means both $\mathcal{\widetilde{A}}_+$ and $\mathcal{\widetilde{A}}_-$ are standard convex.  Since we can assume that after thickening $\mathcal{A}_{r}$ is still standard convex, this means that in order for $\textrm{slope}(\Gamma_{\partial N_{r}})$ to change, the holonomy of $\Gamma_{\mathcal{A}_{r}}$ must have changed.  But this will result in a change in $\textrm{slope}(\Gamma_{\partial N_{r+1}})$, since $\mathcal{A}_ {L_{r}}$ stays fixed and any change in holonomy of $\Gamma_{\mathcal{\widetilde{A}}_+}$ and $\Gamma_{\mathcal{\widetilde{A}}_-}$ cancels each other out and does not affect $\textrm{slope}(\Gamma_{\partial N_{r+1}})$.  Thus we would have a slope-changing thickening of $N_{r+1}$, which by hypothesis cannot occur.

\textbf{Case 2:}  After thickening, suppose $\mathcal{\widetilde{A}}$ is no longer standard convex.  Now note that there are no bypasses on the $\partial N(L_{r})$-edge of $\mathcal{\widetilde{A}}$; furthermore, any bypass for $\mathcal{\widetilde{A}}_+$ on the $\partial N_{r}$-edge must be cancelled out by a corresponding bypass for $\mathcal{\widetilde{A}}_-$ on the $\partial N_{r}$-edge as in part (a) of Figure \ref{nonthickening2}, so as not to induce a bypass on the $\partial N(L_{r})$-edge as in part (b) of the same figure.  But then again, in order for $\textrm{slope}(\Gamma_{\partial N_{r+1}})$ to remain constant, the holonomy of $\Gamma_{\mathcal{A}_{r}}$ must remain constant, and thus $\textrm{slope}(\Gamma_{\partial N_{r}})$ must also have remained constant, with just an increase in the number of dividing curves.

This proves the claim that $N_{r}$ does not thicken, and we therefore know that its boundary slope is $-(k+1)/(A_{r}k+B_{r})$.  Furthermore, we know the number of dividing curves is $2n$ where $n \geq n_r^k$.  Suppose, for contradiction, that $n > n_r^k$.  Then we know we can thicken $N_{r}$ to an $N_{r}^{k}$, and if we take a convex annulus from $\partial N_{r}$ to $\partial N_{r}^{k}$ whose boundaries are $(p_{r+1},q_{r+1})$ rulings, by the Imbalance Principle there must be bypasses on the $\partial N_{r}$-edge.  But these would induce bypasses off of $\infty'$ rulings on $N_{r+1}$, which by hypothesis cannot exist.  Thus $n = n_r^k$, and by a calculation as above we obtain that the intersection boundary slope of $N_{r+1}$ must be $-(k'+1)/(A_{r+1}k'+B_{r+1})$ for the integer $k'=k/q_{r+1}$.
\end{proof}

We now finish the proof of item 2 of the inductive hypothesis.

\begin{lemma}
If $P_{r+1}/q_{r+1} > w(K_r)$, then $w(K_{r+1})=\overline{\tb}(K_{r+1})$.
\end{lemma}

\begin{proof}
Using the above Lemma \ref{nonthickening inductive step}, we need to show that $1/\overline{t}(K_{r+1}) < -(k'+1)/(A_{r+1}k'+B_{r+1})$ for any candidate $N_{r+1}^{k'}$.  But changing to standard $\mathcal{C}$ coordinates, this means we need to show that $1/\overline{\tb}(K_{r+1}) < 2/(A_{r+1} - B_{r+1})$.  By Lemma \ref{overlinetbcalc}, this is true if and only if

\begin{equation}
\label{inequalitytoprove}
A_{r+1}-B_{r+1} < 2[A_{r+1}-(P_{r+1} - q_{r+1}w(K_r))]
\end{equation}

We know inductively that $1/w(K_r) < 2/(A_r-B_r)$.  We use this fact below, along with the identities in equation \ref{identities}, to prove inequality \ref{inequalitytoprove}.  We begin with the right hand side:

\begin{eqnarray}
2[A_{r+1}-(P_{r+1} - q_{r+1}w(K_r))] &=& 2q_{r+1}w(K_r) - 2P_{r+1} + 2A_{r+1}\nonumber \\
  &>& q_{r+1}A_r -q_{r+1}B_r - 2P_{r+1} + 2A_{r+1}\nonumber \\
  &=& q_{r+1}A_r -q_{r+1}B_r - 2(q_{r+1}A_r +p_{r+1}) + 2A_{r+1}\nonumber \\
  &=& -q_{r+1}A_r -q_{r+1}B_r - p_{r+1} -p_{r+1} +2A_{r+1}\nonumber \\
  &=&  -q_{r+1}A_r - B_{r+1} - p_{r+1} +2A_{r+1}\nonumber \\
  &=& A_{r+1} - B_{r+1} + [A_{r+1} - (q_{r+1}A_r + p_{r+1})]\nonumber \\
  &>& A_{r+1} - B_{r+1}\nonumber
\end{eqnarray}
\end{proof}

We conclude this subsection by proving item 5 of the inductive hypothesis, using a construction and two lemmas.  We begin with the construction, which shows that the candidate $N_{r+1}^{k'}$ exist for $k' \geq C_{r+1}$, where $C_{r+1}$ is some positive integer.

\bigskip

\begin{construction}
\label{main construction}
We know inductively that there exists a $C_r$ such that if $k\geq C_r$, then the $N_r^k$ exist and fail to thicken, and have convex meridian discs with bypasses all of the same sign.  So suppose $k/q_{r+1} \in \mathbb{N}$ for some $k \geq C_r$.  We will show that $N_{r+1}^{k'}$ exists for $k':=k/q_{r+1}$.  Then $C_{r+1}$ will be the least such $k/q_{r+1} \in \mathbb{N}$.

The idea is to build $S^3$.  We first take one of the two universally tight candidate $N_{r+1}^{k'}$, with intersection boundary slope $-(k'+1)/(A_{r+1}k'+B_{r+1})$, and with convex meridian discs having bypasses all of the same sign; thus the two possible contact structures on $N_{r+1}^{k'}$ differ by $-id$.  We then show that we can use such a $N_{r+1}^{k'}$ to build $N_r^{k'q_{r+1}}$, essentially running backwards the decomposition from Lemma \ref{nonthickening inductive step} above.  To this end, let $\cA$ be a standard convex annulus joining two $\infty'_{r+1}$-longitudes on $\partial N_{r+1}^{k'}$, so that if we then let $R = N_{r+1}^{k'} \cup N(\cA)$, we have that $R$ is diffeomorphic to $T^2 \times [0,1]$, with a $[0,1]$-invariant contact structure on $N(\cA)$.  Furthermore, we can think of $R$ as containing a horizontal annulus joining $T^2 \times \{0\}$ to $T^2 \times \{1\}$, and such that the original $\infty'_{r+1}$-longitudes on $\partial N_{r+1}^{k'}$ intersect this horizontal annulus $q_{r+1}$ times; thus, with an appropriate choice of $\infty'_r$-longitude for $T^2 \times \{i\}$, the original $\infty'_{r+1}$-longitudes on $\partial N_{r+1}^{k'}$ are now $(p_{r+1},q_{r+1})$ curves on $T^2 \times \{i\}$.  

We will thus think of $R$ as fibering over the horizontal annulus with fiber circles representing the knot type $K_{r+1}$.  For either choice of the two universally tight contact structures on $N_{r+1}^{k'}$, the contact structure on $R$ can be isotoped to be transverse to the fibers of $R$, while preserving the dividing set on $R$. Hence the contact structure is horizontal, and therefore universally tight.  Furthermore, with appropriately chosen dividing curves on $\cA$, we can obtain intersection boundary slopes (on the two boundary tori $T^2 \times \{0\}$ and $T^2 \times \{1\}$) of $-(k'q_{r+1}+1)/(A_rk'q_{r+1}+B_r)$ and $-1/(p_{r+1}k'+A_{r}k + B_{r}))$; i.e., the intersection boundary slopes of a $N_{r}^{k'q_{r+1}}$ and a Legendrian of twisting $-(p_{r+1}k'+A_{r}k + B_{r})$.

We now glue, onto the one side of $R$, a standard neighborhood of a Legendrian $L_r$ of twisting $-(p_{r+1}k'+A_{r}k + B_{r})$; we claim the resulting solid torus $N_r$ is one of the $N_r^{k'q_{r+1}}$.  To see this, look at a $q_{r+1}$-fold cover of $N_r$, and examine its convex meridian disc $D_r$ (which is also the same convex meridian disc $D_r$ for the $N_r$ downstairs).  The disc $D_r$ is formed by taking $q_{r+1}$ meridian discs from the $q_{r+1}$-copies of lifts of $N_{r+1}^{k'}$, and first banding them together via bands coming from the $[0,1]$-invariant $N(\cA)$, and then finally gluing in the convex meridian disc for the standard neighborhood of a Legendrian.  But now evaluating the relative Euler class (of $\xi$) on $D_r$, we note that these bands and the meridian disc for the standard neighborhood yield no obstruction, and thus we obtain $\pm k'q_{r+1}$, as each of the $q_{r+1}$ meridian discs from $N_{r+1}^{k'}$ yields $\pm k'$.  

We then know inductively that this $N_r^{k'q_{r+1}}$ (and hence $N_{r+1}^{k'}$) exists in $S^3$.

\raggedleft\qedsymbol
\end{construction}

\bigskip

We now show that the $N_{r+1}^{k'}$ coming from the above construction in fact fail to thicken.

\begin{lemma}
The $N_{r+1}^{k'}$ from Construction \ref{main construction} fail to thicken for $k' \geq C_{r+1}$. 
\end{lemma}

\begin{proof}
To show that $N_{r+1}^{k'}$ fails to thicken, by Lemmas \ref{thickeningorder} and \ref{nonthickening inductive step} it suffices to show that $N_{r+1}^{k'}$ does not thicken to any $N_{r+1}^{k''}$, where $k'' < k'$.  Inductively, we can assume $N_{r}^k$ fails to thicken for $k \geq C_r$; in particular, the $N_{r}^{k'q_{r+1}}$ that contains $N_{r+1}^{k'}$ fails to thicken.  So let $k = k'q_{r+1}$.  Then define $M_r^k = S^3 \backslash N_r^k$, and define $M_{r+1}^{k'}=S^3 \backslash N_{r+1}^{k'}$.

We first make some purely topological observations, which in the rest of this proof we will refer to as {\em the topological observations}.  We begin by observing that $K_{r+1}$ is a fibered knot, and has periodic monodromy -- see, for example, \cite{[A]}.  One way to see this is as follows.  We think of $K_{r+1}$ embedded on $\partial N_r$, and let $\Sigma_{r+1}$ be a Seifert surface for $K_{r+1}$.  Furthermore, we note that $\Sigma_{r+1}$ can be formed by taking $q_{r+1}$ copies of the Seifert surface $\Sigma_r$ for the Seifert longitude on $\partial N_r$, and $P_{r+1}$ copies of a meridian disc $D_r$ for $N_r$, and banding them together with $P_{r+1}q_{r+1}$ positive (half-twist) bands.  We then observe that if we take a slightly larger $N_r' \supset N_r$, there will be $q_{r+1}$ separating simple closed curves on $\Sigma_{r+1}$ that are in fact preferred Seifert longitudes for $\partial N_r'$, and thus bound Seifert surfaces $\Sigma_r$ for the knot $K_r$ in the complement of $N_r'$ (all $q_{r+1}$ of which are subsurfaces of $\Sigma_{r+1}$).  In fact, the monodromy for $\Sigma_{r+1}$ is reducible along these $q_{r+1}$ curves; that is, if we call $\Sigma_{r+1} \cap N_r' := \sigma_{r+1}$, the monodromy will take $\sigma_{r+1}$ to itself, and sweep out the interior of $N_r'$.  Moreover, the monodromy will restrict to being periodic on $\sigma_{r+1}$, of period $P_{r+1}q_{r+1}$, as repeated application of the monodromy cycles through the $P_{r+1}q_{r+1}$ bands.  Then, since positive torus knots have periodic monodromy, inductively we can assume that the exterior of $N_r'$ fibers periodically with the $q_{r+1}$ copies of the $\Sigma_r$'s.  As a result, there is a positive integer $m_{r+1}$ such that $\phi^{m_{r+1}} = id$ (where here $\phi$ is the $\Sigma_{r+1}$-monodromy), and such that $P_{r+1}q_{r+1}$ divides $m_{r+1}$.  

We return to contact topology, and now let $\Sigma_{r+1}$ be a Seifert surface for a preferred longitude on $\partial N_{r+1}^{k'}$; so $\Sigma_{r+1}$ is a surface of genus $g'$ with one boundary component.  As noted in the topological observations, there are $q_{r+1}$ separating simple closed curves on $\Sigma_{r+1}$ that are in fact preferred longitudes for $\partial N_r^k$, and thus bound Seifert surfaces $\Sigma_r$ for the knot $K_r$.  We will call the genus of such a Seifert surface $\Sigma_r$, $g$.  Also we will call $\Sigma_{r+1} \cap N_r^k :=\sigma_{r+1}$; so $\Sigma_{r+1} = \sigma_{r+1} \cup \left(\bigcup_{j=1}^{q_{r+1}} \Sigma_r^j\right)$.

We look at finite covers of $M_{r+1}^{k'}$ that are obtained by cutting $M_{r+1}^{k'}$ along $\Sigma_{r+1}$ and then cyclically stacking copies of these split-open $M_{r+1}^{k'}$.  We first look at a $P_{r+1}q_{r+1}$-fold cover obtained in this fashion, and, due to the topological observations above, focus in on the lift of the space $N_r^k \setminus N_{r+1}^{k'}$ which contains $\sigma_{r+1}$.  If we arrange that downstairs $\partial N_r^k$ has Legendrian rulings that are $(P_{r+1},q_{r+1})$ cables (which are $\infty_{r+1}^\prime$-rulings on $\partial N_{r+1}^{k'}$), then upstairs, in the $P_{r+1}q_{r+1}$-fold cover, the lift of $N_r^k \setminus N_{r+1}^{k'}$ can be fibered by Legendrian fibers all with twisting $-(A_{r+1}k'+B_{r+1})$.  The reason for this is as follows.  First of all, the $\infty_{r+1}^\prime$-rulings have twisting $-(A_{r+1}k'+B_{r+1})$ on $\partial N_{r+1}^{k'}$, and intersect the $\infty_{r+1}$-longitude positively $P_{r+1}q_{r+1}$ times; hence upstairs in the $P_{r+1}q_{r+1}$-fold cover they will lift to Legendrians of twisting $-(A_{r+1}k'+B_{r+1})$.  As a result, the standard convex annulus $\cA$ from Construction \ref{main construction} will be fibered by Legendrians of twisting $-(A_{r+1}k'+B_{r+1})$ upstairs in the cover as well.  Moreover, the $(P_{r+1},q_{r+1})$ rulings on $\partial N(L_r)$ in Construction \ref{main construction} also have twisting $-(A_{r+1}k'+B_{r+1})$, and in the cover will become longitudinal $(P_{r+1},1)$ rulings, but still with twisting $-(A_{r+1}k'+B_{r+1})$.  Furthermore, the lift of $N(L_r)$ will have convex boundary with two longitudinal dividing curves (a different longitude, of course).  Thus we see that the contact structure on this lift of $N(L_r)$ is just a standard neighborhood of one of the ruling curves (pushed into the interior of the solid torus), and thus the solid torus can be fibered by Legendrians of twisting $-(A_{r+1}k'+B_{r+1})$. 

Note that the rest of the cover (outside the lift of $N_r^k \setminus N_{r+1}^{k'}$) is fibered (horizontally) by the copies of the $\Sigma_r$'s.  By the proof of Lemma 3.3 in \cite{[ELT]} for the case of positive torus knots (see also the discussion in Lemma \ref{basecasepositivetorusknots} above), inductively we can assume that the monodromy for the fibered space $M_r^k$ is periodic, with period that divides a positive integer $m_r$, and such that a resulting $m_r$-fold product cover can be fibered by Legendrian fibers that all have twisting $-s_r(A_rk+B_r)$, where $s_r$ is again some positive integer (for positive torus knots, $m_1 = P_1q_1$ and $s_1 = 1$).  It will be convenient for us, however, to take $m_r$, and multiply it by $-\chi(K_r)$ to get a new $m_r$; in other words, we can assume that $-\chi(K_r)$ divides $m_r$ and $s_r$, and we will still have the $m_r$-fold product cover of $M_r^k$ being fibered by Legendrians all having twisting $-s_r(A_rk+B_r)$.

As a consequence of this and the above topological observations, we can now cyclically stack $m_r$ copies of our $P_{r+1}q_{r+1}$-fold cover of $M_{r+1}^{k'}$ to obtain $\widetilde{M}_{r+1}^{k'} = S^1 \times \Sigma_{r+1}$.  Furthermore, if we restrict to $S^1 \times \sigma_{r+1} \subset S^1 \times \Sigma_{r+1}$, the space $S^1 \times \sigma_{r+1}$ can be fibered by Legendrians all of twisting $-s_{r+1}(A_{r+1}k'+B_{r+1})$, for some positive integer $s_{r+1}$, with respect to the product framing.  However, at the moment all we know is that the $q_{r+1}$ copies of $S^1 \times \Sigma_{r}$ can be fibered by {\em topological} copies of these Legendrian fibers in $S^1 \times \sigma_{r+1}$; what we will show is that in fact $S^1 \times \Sigma_{r+1}$ can be fibered by Legendrian copies of the fibers in $S^1 \times \sigma_{r+1}$.

To this end, we first establish some notation; downstairs let $T = \partial N_r^k$.  As just mentioned, we may assume that the rulings on $T$ are copies of $\infty'_{r+1}$ (i.e., $(P_{r+1},q_{r+1})$ cables on $T$), and the space $N_r^k \setminus N_{r+1}^{k'}$ bounded by $T$ lifts to $S^1 \times \sigma_{r+1}$, where all the $S^1$ fibers are Legendrian isotopic to lifts of $\infty'_{r+1}$, and have twisting $-s_{r+1}(A_{r+1}k'+B_{r+1})$ for some positive integer $s_{r+1}$.  We will call these $S^1$ fibers $S_{r+1}^1$, and note that they are topologically isotopic to the $S^1$ fibers in the product space $S^1 \times \Sigma_{r+1}$.  We also have that if we think of $M_r^k$ as bounded by $T$, then $M_r^k$ lifts to $q_{r+1}$ copies of $S^1 \times \Sigma_r$, where all the $S^1$ fibers are Legendrian isotopic to lifts of $\infty'_r$, and have twisting $-s_rP_{r+1}(A_rk+B_r)$.  We will call these $S^1$ fibers $S_r^1$, and emphasize that these are not the same as the $S_{r+1}^1$'s.  However, we will show that in fact, all of $\widetilde{M}_{r+1}^{k'}$ can be fibered by Legendrian $S^1_{r+1}$'s. 

On the Seifert surface $\Sigma_{r+1}$, we will label the $q_{r+1}$ $\Sigma_r$'s as $\Sigma_r^j$.  Now let $\alpha_{r+1}^i$ be $2g'$ disjoint arcs on $\Sigma_{r+1}$, each with endpoints on $\partial \Sigma_{r+1}$, and such that if we cut along the $\alpha_{r+1}^i$ we obtain a polygon $P_{r+1}$.  Also, let $\alpha_{r,j}^i$ be $2g$ disjoint arcs on $\Sigma_r^j$ that, when we cut along them, yield polygons $P_r^j$.  Thus we have solid tori $S_r^1 \times P_r^j$ embedded in $\widetilde{M}_{r+1}^{k'}$.  We can calculate the boundary slopes of these solid tori using the framing coming from the lifts of $\infty_r^\prime$; this calculation is similar to that in Lemma 3.3 in \cite{[ELT]}.  Specifically, note that a longitude for this torus intersects $\Gamma$, $2s_rP_{r+1}(A_rk+B_r)$ times, and a meridian for this torus is composed of $2$ copies each of the associated $2g$ arcs $\alpha_{r,j}^i$, as well as $4g$ arcs $\beta_i$ from $\partial \Sigma_r^j$.  Now since $\partial \Sigma_r^j$ is a preferred longitude downstairs in $M_r^k$, we know that $\Gamma$ intersects these $\beta_i$, $2(-\chi(K_r)) = 2(2g-1)$ times positively; see Remark \ref{eulercharremark} above.  But then the edge-rounding that results at each intersection of an $S_r^1 \times \beta_i$ with an $S_r^1 \times \alpha_{r,j}^i$ yields $4g$ negative intersections with $\Gamma$.  Thus we obtain after edge-rounding that the boundary slope is $-1/(s_rP_{r+1}(A_rk+B_r))$; as a consequence, we see that the solid torus $S_r^1 \times P_r^j$ is simply a standard neighborhood of a Legendrian of twisting $-(s_rP_{r+1}(A_rk+B_r))$.

Now we switch our attention to the $S_{r+1}^1$'s, and note that the arcs $\alpha_{r+1}^i$ that stay in $\sigma_{r+1}$ represent an interval's worth of $S^1_{r+1}$ fibers of twisting $-s_{r+1}(A_{r+1}k'+B_{r+1})$, and hence represent standard convex annuli in the space $\widetilde{M}_{r+1}^{k'}$.  The arcs $\alpha_{r+1}^i$ that leave $\sigma_{r+1}$ represent convex annuli that are fibered by Legendrian $S^1_{r+1}$'s only when restricted to their intersection with the lift of the space $N_r^k \setminus N_{r+1}^{k'}$ bounded by $T$.  So what is of interest is a convex annulus $\cA_i$ with boundary components that both have twisting $-s_{r+1}(A_{r+1}k'+B_{r+1})$, fibered by topological copies of the $S_{r+1}^1$'s but which is embedded in one of the $q_{r+1}$ lifts of $M_r^k$.

So suppose, for contradiction, that there exists a bypass on one of the $\cA_i$'s.  We look at what passing through this bypass will do on the lift of $T$ to which $\cA_i$ is attached; we use the framing on the lift of $T$ that comes from the the lifts of $\infty'_r$.  First, recall that we know that $q_{r+1}/p_{r+1}$ is clockwise from $-(k+1)/(A_rk+B_r)$ in the Farey tessellation; as a result, we know that the bypass of interest is on a ruling with slope $1/t'$ that is clockwise (in the Farey tessellation) from the dividing slope $s$ of the lift of $T$.  Moreover, we know what this dividing slope $s$ is; it is $-\chi(K_r)/(-s_rP_{r+1}(A_rk+B_r))$, since the original preferred Seifert longitude on $T$ lifts to the meridian on the lift of $T$.  But, since $-\chi(K_r)$ divides $s_r$, this means in lowest terms, $s = -1/t$.  As a result, passing through the bypass would yield a new torus $T'$, on which is a longitudinal curve $\gamma$ topologically isotopic to the $S_r^1$'s, but with twisting greater than $-s_rP_{r+1}(A_rk+B_r)$.  But if we then split the $S_r^1 \times \Sigma_r^j$ that contains the $\cA_i$ along arcs $\alpha_{r,j}^i$ to obtain $S_r^1 \times P_r^j$, and then pass to a finite cover of the base by tiling copies of $S_r^1 \times P_r^j$ (similar to what we did in Lemma 3.3 in \cite{[ELT]}), we will enclose $\gamma$ in a standard neighborhood of a Legendrian of twisting $-s_rP_{r+1}(A_rk+B_r)$, which is a contradiction.  Thus $\cA_i$ must be standard convex.


As a consequence, if we now use the product framing coming from the $S_{r+1}^1$'s, and now split the whole $\Sigma_{r+1}$ along all arcs using the standard convex annuli $S_{r+1}^1 \times \alpha_{r+1}^i$ to obtain $S_{r+1}^1 \times P_{r+1}$, then that boundary torus will have a characteristic foliation that matches that of the standard neighborhood of a Legendrian with twisting $-s_{r+1}(A_{r+1}k'+B_{r+1})$, since the dividing curves on the lift of $\partial N_{r+1}^{k'}$ intersect $\partial \Sigma_{r+1}$ exactly $2(-\chi(K_{r+1}))$ times and hence a similar edge-rounding calculation applies as above.  As a result, the contact structure can be isotoped so that all of the $S_{r+1}^1$ fibers in $\widetilde{M}_{r+1}^{k'}$ are Legendrian of twisting $-s_{r+1}(A_{r+1}k'+B_{r+1})$.  Thus the argument that $N_{r+1}^{k'}$ fails to thicken proceeds exactly as in Lemma 3.3 in \cite{[ELT]}; specifically, if $N_{r+1}^{k'}$ thickens, then there exists a curve $\gamma'$ upstairs in $\widetilde{M}_{r+1}^{k'}$ which is topologically isotopic to the $S_{r+1}^1$'s but has greater twisting.  However, if we then split the $\Sigma_{r+1}$ along arcs $\alpha_{r+1}^i$ to obtain $S_{r+1}^1 \times P_{r+1}$, and then pass to a finite cover of the base by tiling copies of $S_{r+1}^1 \times P_{r+1}$ (similar to what we did in Lemma 3.3 in \cite{[ELT]}), we will enclose $\gamma'$ in a standard neighborhood of a Legendrian of twisting $-s_{r+1}(A_{r+1}k'+B_{r+1})$, which is a contradiction. 
\end{proof}

We conclude with the following lemma that calculates rotation numbers.

\begin{lemma}
\label{calculate rotation numbers}
Any non-thickenable $N_{r+1}^{k'}$ have contact structures that are universally tight and have convex meridian discs $D$ whose bypasses bound half-discs all of the same sign; i.e., $\rot(\partial D) = \pm k'$.  Also, the preferred longitude on $\partial N_{r+1}^{k'}$ has rotation number zero for $k'>0$.
\end{lemma}  

\begin{proof}
We first prove that the contact structure on a candidate $N_{r+1}^{k'}$ which fails to thicken is universally tight.  To see this note that from Lemma \ref{nonthickening inductive step} above, and the inductive hypothesis, such a candidate $N_{r+1}^{k'}$ is embedded inside a $N_r^k$ with a universally tight contact structure.  Now there is a $q_{r+1}$-fold cover of $N_r^k$ that maps a total of $q_{r+1}$ lifts $\widetilde{N}_{r+1}^{k'}$ to $N_{r+1}^{k'}$, the lifts themselves each being an $S^1 \times D^2$ .  This cover in turn has a universal cover $\mathbb{R} \times D^2$ that contains $q_{r+1}$ copies of a universal cover $\mathbb{R} \times D^2$ of $N_{r+1}^{k'}$.  Since, by the inductive hypothesis, the universal cover of $N_r^k$ has a tight contact structure, a tight contact structure is thus induced on the universal cover of $N_{r+1}^{k'}$.

Now recall that $N_r^k$ is formed from $N_{r+1}^{k'}$ by first joining $\infty'$-longitudes on $\partial N_{r+1}^{k'}$ with an annulus $\cA$ to get a thickened torus $R = T^2 \times [0,1]$, and then gluing in a standard neighborhood of a Legendrian knot.  Thus, since $N_r^k$ has bypasses all of the same sign, by similar reasoning as that in Construction \ref{main construction}, it follows that a horizontal annulus in $R$ has bypasses all of the same sign.  We will thus have that $N_{r+1}^{k'}$ must have convex meridian discs all of the same sign.  The computation of rotation numbers for the meridian curve follows.  

To show that the preferred longitude on $\partial N_{r+1}^{k'}$ has rotation number zero, we use an argument similar to that used in Lemma 3.4 in \cite{[ELT]}.  We call the meridian disc for $N_r^k$, $D_r$, and the Seifert surface for the preferred longitude on $\partial N_r^k$, $\Sigma_r$.  If we then look at the $(P_{r+1},q_{r+1})$ cable on $\partial N_r^k$, we can calculate its rotation number as

\begin{equation}
\rot((P_{r+1},q_{r+1}))=P_{r+1}\rot(\partial D_r) + q_{r+1}\rot(\partial \Sigma_r) = P_{r+1}(\pm q_{r+1}k')\nonumber
\end{equation}

\noindent But then since this same knot is a $(P_{r+1}q_{r+1},1)$ cable on $\partial N_{r+1}^{k'}$, we have that $\rot((P_{r+1},q_{r+1}))$ $= P_{r+1}q_{r+1}(\pm k') + q_{r+1} \rot(\partial \Sigma)$, where $\Sigma$ is a Seifert surface for the preferred longitude on $\partial N_{r+1}^{k'}$.  This implies that $\rot(\partial \Sigma)=0$.
\end{proof}  

\medskip

\subsection{Case II:  $w(K_r) > P_{r+1}/q_{r+1} > 0$.}   As in Case I, we work through proving items 2-5 in the inductive hypothesis via a series of lemmas.  

We begin by proving item 2 in the inductive hypothesis.

\begin{lemma}
If $w(K_r) > P_{r+1}/q_{r+1} > 0$, then $\overline{\tb}(K_{r+1})=w(K_{r+1})=A_{r+1}$.
\end{lemma}

\begin{proof}
The proof is almost identical to that of step 1 in the proof of Theorem 1.5 from Section 6 of \cite{[L]}; we will include the details, though, since certain aspects differ.  We first examine representatives of $K_{r+1}$ at $\overline{\tb}$.  Since there exists a convex torus representing $K_r$ with Legendrian divides that are $(p_{r+1},q_{r+1})$ cablings (inside of the solid torus representing $K_r$ with $\textrm{slope}(\Gamma)=1/\overline{t}(K_r)$) we know that $\overline{\tb}(K_{r+1}) \geq P_{r+1}q_{r+1}=A_{r+1}$.  To show that $\overline{\tb}(K_{r+1}) = A_{r+1}$, we show that $\overline{t}(K_{r+1})=0$ by showing that the contact width $w(K_{r+1},\mathcal{C}')=0$, since this will yield $\overline{\tb}(K_{r+1}) \leq w(K_{r+1})=A_{r+1}$.  So suppose, for contradiction, that some $N_{r+1}$ has convex boundary with $\textrm{slope}(\Gamma_{\partial N_{r+1}})=s > 0$, as measured in the $\mathcal{C}'$ framing, and two dividing curves.  After shrinking $N_{r+1}$ if necessary, we may assume that $s$ is a large positive integer.  Then let $\mathcal{A}$ be a convex annulus from $\partial N_{r+1}$ to itself having boundary curves with slope $\infty'$.  Taking a neighborhood of $N_{r+1} \cup \mathcal{A}$ yields a thickened torus $R$ with boundary tori $T_1$ and $T_2$, arranged so that $T_1$ is inside the solid torus $N_r$ representing $K_r$ bounded by $T_2$.

Now there are no boundary parallel dividing curves on $\mathcal{A}$, for otherwise, we could pass through the bypass and increase $s$ to $\infty'$, yielding excessive twisting inside $N_{r+1}$.  Hence $\mathcal{A}$ is in standard form, and consists of two parallel nonseparating arcs.  We now choose a new framing $\mathcal{C}''$ for $N_r$ where $(p_{r+1},q_{r+1}) \mapsto (0,1)$; then choose $(p'',q'') \mapsto (1,0)$ so that $p''q_{r+1}-q''p_{r+1}=1$ and such that $\textrm{slope}(\Gamma_{T_1})=-s$ and $\textrm{slope}(\Gamma_{T_2})=1$.  As mentioned in \cite{[EH1]}, this is possible since $\Gamma_{T_1}$ is obtained from $\Gamma_{T_2}$ by $s+1$ right-handed Dehn twists.  Then note that in the $\mathcal{C'}$ framing, we have that $q_{r+1}/p_{r+1} > \textrm{slope}(\Gamma_{T_2})=(q''+q_{r+1})/(p''+p_{r+1}) > q''/p''$, and $q_{r+1}/p_{r+1}$ and $q''/p''$ are connected by an arc in the Farey tessellation of the hyperbolic disc (see section 2.2.3 in \cite{[ELT]}).  Thus, since $1/\overline{t}(K_r)$ is connected by an arc to $0/1$ in the Farey tessellation, we must have that $(q''+q_{r+1})/(p''+p_{r+1}) > 1/\overline{t}(K_r)$.  Thus we can thicken $N_r$ to one of the solid tori with $\textrm{slope}(\Gamma)=-(k+1)/(A_rk+B_r)$ which fails to thicken.  Then, just as in Claim 4.2 in \cite{[EH1]}, we have the following:

\begin{itemize}
	\item[(i)] Inside $R$ there exists a convex torus parallel to $T_i$ with slope $q_{r+1}/p_{r+1}$; 
	\item[(ii)] $R$ can thus be decomposed into two layered {\em basic slices};
	\item[(iii)] The tight contact structure on $R$ must have {\em mixing of sign} in the Poincar$\acute{\textrm{e}}$ duals of the relative half-Euler classes for the layered basic slices;
	\item[(iv)] This mixing of sign cannot happen inside the universally tight solid torus which fails to thicken.
\end{itemize}

This last statement is due to the proof of Proposition 5.1 in \cite{[H1]}, where it is shown that mixing of sign will imply an overtwisted disc in the universal cover of the solid torus.  Thus we have contradicted $s > 0$.  So $\overline{\tb}(K_{r+1})=P_{r+1}q_{r+1}=A_{r+1}$.
\end{proof}

With the following lemma we prove that items 3 and 4 of the inductive hypothesis hold for $K_{r+1}$.

\begin{lemma}
\label{nonthickening inductive step 2}
If $w(K_r) > P_{r+1}/q_{r+1} > 0$, let $N_{r+1}$ be a solid torus representing $K_{r+1}$, for $r \geq 1$.  Then $N_{r+1}$ can be thickened to an $N_{r+1}^{k'}$ for some nonnegative integer $k'$.  Moreover, if $N_{r+1}$ fails to thicken, then it has the same boundary slope as some $N_{r+1}^{k'}$, as well as at least $2n_{r+1}^{k'}$ dividing curves.
\end{lemma}

\begin{proof}
This is the case where $p_{r+1} < 0$ but $q_{r+1}/p_{r+1} \in (1/\overline{t}(K_r), -1/A_r)$; we have that $\overline{t}(K_{r+1})=0$.  We begin as we did in Case I.  If $N_{r+1}$ is a solid torus representing $K_{r+1}$, as before choose $L_r$ in $S^3 \backslash N_{r+1}$ such that $\partial N(L_r)$ is joined to $\partial N_{r+1}$ by an annulus $\mathcal{A}_{(p_{r+1},q_{r+1})}$, and with $\tb(L_r)$ maximized over topological isotopies in the space $S^3 \backslash N_{r+1}$.

Now suppose $\textrm{slope}(\Gamma_{\partial N(L_{r})})=-1/m$ where $-1/m < q_{r+1}/p_{r+1}$; we know $m \geq 0$.  Then inside $N(L_r)$ is an $N_r$ with boundary slope $q_{r+1}/p_{r+1}$.  But then we can extend $\mathcal{A}_{(p_{r+1},q_{r+1})}$ to an annulus that has no twisting on one edge, and we can thus thicken $N_{r+1}$ so it has boundary slope $\infty'$.  Moreover, since there is twisting inside $N(L_r)$, we can assure there are two dividing curves on the thickened $N_{r+1}$ (see Claim 4.1 in \cite{[EH1]}).  So this situation yields no nontrivial solid tori $N_{r+1}$ which fail to thicken; in other words, $N_{r+1}$ can be thickened to an $N_{r+1}^0$.

Alternatively, suppose $-1/m > q_{r+1}/p_{r+1}$; note here we must have $m > 0$.  Furthermore, for the moment suppose $-1/(m-1) > q_{r+1}/p_{r+1}$.  Then we can use the twist number lemma (Lemma \ref{twistnumber} above) to conclude that there are no bypasses on the $\partial N(L_r)$-edge of $\mathcal{A}_{(p_{r+1},q_{r+1})}$, and so we can thicken $N_{r+1}$ through bypasses so that $\mathcal{A}_{(p_{r+1},q_{r+1})}$ is standard convex.  

Then, as in Lemma \ref{nonthickening inductive step}, we let $N_r := N_{r+1} \cup N(\cA_{(p_{r+1},q_{r+1})}) \cup N(L_r)$.  We know that $w(K_{r+1}, \mathcal{C}') = 0$, and we know that the geometric intersection of the $\infty'$-rulings on $\partial N_{r+1}$ with $\Gamma_{\partial N_{r+1}}$ equals $p_{r+1}+mq_{r+1} > 0$.  Thus, we must have that the $(p_{r+1},q_{r+1})$-rulings intersect $\Gamma_{\partial N_r}$ positively; i.e., $q_{r+1}/p_{r+1}$ is clockwise from $\textrm{slope}(\Gamma_{\partial N_r})$ in the Farey tessellation.  As a result, when we thicken to $N_{r}^k$ as in Lemma \ref{nonthickening inductive step}, we must also have $q_{r+1}/p_{r+1}$ clockwise from $-(k+1)/(A_rk+B_r)$ in the Farey tessellation, for otherwise we could destabilize $L_r$ (in the complement of $N_{r+1}$) via an annulus with $(p_{r+1},q_{r+1})$-ruling boundary on $\partial N(L_r)$, and $(p_{r+1},q_{r+1})$-Legendrian divide boundary on a torus $N_r^\prime$ with boundary slope $q_{r+1}/p_{r+1}$ in the thickened torus cobounded by $\partial N_r$ and $\partial N_r^k$.

Thus, since $q_{r+1}/p_{r+1}$ is clockwise from both $-1/m$ and $-(k+1)/(A_rk+B_r)$ in the Farey tessellation, the calculation of the boundary slope goes through as above in Lemma \ref{nonthickening inductive step} -- see the comment after equation \ref{equationarray} above, and note that such $N_r^k$ exist since $-(k+1)/(A_rk+B_r) \rightarrow -1/A_r$ as $k$ increases.  We conclude that $N_{r+1}$ thickens to some $N_{r+1}^{k'}$.

For the remaining case, suppose $-1/m > q_{r+1}/p_{r+1}$ and $m$ is the least positive integer satisfying this inequality.  Thus $-1/(m-1) < q_{r+1}/p_{r+1}$.  Again look at the $\partial N(L_r)$-edge of $\mathcal{A}_{(p_{r+1},q_{r+1})}$.  We claim that this edge has no bypasses.  So, for contradiction, suppose it does.  Then we can thicken $N(L_r)$ to a solid torus where the (efficient) geometric intersection of $(p_{r+1},q_{r+1})$ with dividing curves is less than $p_{r+1}+mq_{r+1}$.  Suppose the slope of this new solid torus is $-\lambda/\mu < -1/m$, where $\lambda > 1$ since $m$ is minimized in the complement of $N_{r+1}$.

We do some calculations.  Note first that if $m/\mu > 1$, then $m > \mu$, which means $m-1 \geq \mu$, which implies $-1/(m-1) \geq -1/\mu > -\lambda/\mu$, which cannot happen, again since $m$ is minimized in the complement of $N_{r+1}$.  Thus we must have $m/\mu \leq 1$.  But then the geometric intersection of $(p_{r+1},q_{r+1})$ with $(-\mu,\lambda)$ is $\lambda p_{r+1} + \mu q_{r+1} > (\mu/m)p_{r+1} + \mu q_{r+1} \geq (m/\mu)[(\mu/m)p_{r+1} + \mu q_{r+1}] = p_{r+1} + mq_{r+1}$.  This is a contradiction.

Thus there are no bypasses on the $\partial N(L_r)$-edge of $\mathcal{A}_{(p_{r+1},q_{r+1})}$, and we can thicken $N_{r+1}$ through any bypasses so that $\mathcal{A}_{(p_{r+1},q_{r+1})}$ is standard convex.  The calculations that show $N_{r+1}$ thickens to $N_{r+1}^{k'}$ go through as above in Lemma \ref{nonthickening inductive step}; in particular, as above, the non-thickenable $N_r^k$ that is used will be such that $q_{r+1}/p_{r+1} < -(k+1)/(A_rk+B_r)$.

This shows that any $N_{r+1}$ representing $K_{r+1}$ can be thickened to one of the $N_{r+1}^{k'}$, and if $N_{r+1}$ fails to thicken, then it has the same boundary slope as some $N_{r+1}^{k'}$.  We now show that if $N_{r+1}$ fails to thicken, and if it has the minimum number of dividing curves over all such $N_{r+1}$ which fail to thicken and have the same boundary slope as $N_{r+1}^{k'}$, then $N_{r+1}$ is actually an $N_{r+1}^{k'}$.   

To see this, as above we can choose a Legendrian $L_{r}$ that maximizes $\tb$ in the complement of $N_{r+1}$ and such that we can join $\partial N(L_{r})$ to $\partial N_{r+1}$ by a convex annulus $\mathcal{A}_{(p_{r+1},q_{r+1})}$ whose boundaries are $(p_{r+1},q_{r+1})$ and $\infty'$ rulings on $\partial N(L_{r})$ and $\partial N_{r+1}$, respectively.  Now since $N_{r+1}$ fails to thicken, we can assume that $q_{r+1}/p_{r+1} < -1/m$ and that there are no bypasses on the $\partial N(L_{r})$-edge, and in this case we have no bypasses on the $\partial N_{r+1}$-edge since $N_{r+1}$ fails to thicken and is at minimum number of dividing curves.

As above, let $N_{r}:=N_{r+1} \cup N(\mathcal{A}_{(p_{r+1},q_{r+1})}) \cup N(L_{r})$.  We claim this $N_{r}$ fails to thicken -- the proof proceeds identically as above in Lemma \ref{nonthickening inductive step}, as does the proof that $N_{r+1}$ is in fact an $N_{r+1}^{k'}$.
\end{proof}

The following proof of item 5 of the inductive hypothesis is similar to that of Case I.

\begin{lemma}
If $w(K_r) > P_{r+1}/q_{r+1} > 0$, the candidate $N_{r+1}^{k'}$ exist and actually fail to thicken for $k' \geq C_{r+1}$, where $C_{r+1}$ is some positive integer.  Moreover, these $N_{r+1}^{k'}$ have contact structures that are universally tight and have convex meridian discs whose bypasses bound half-discs all of the same sign.  Also, the preferred longitude on $\partial N_{r+1}^{k'}$ has rotation number zero for $k'>0$.
\end{lemma}  

\begin{proof}
The proof that the contact structure on a candidate $N_{r+1}^{k'}$ which fails to thicken is universally tight is identical to the argument in Case I, as are the calculations of the rotation numbers.
  
Now we know inductively that there exists a $C_r$ such that if $k \geq C_r$, then the $N_r^k$ exist and fail to thicken.  So suppose $k/q_{r+1} \in \mathbb{N}$ for some $k \geq C_r$.  Also assume that $q_{r+1}/p_{r+1} < -(k+1)/(A_rk+B_r)$; we know such a $k$ exists since $-(k+1)/(A_rk+B_r) \rightarrow -1/A_r$ as $k$ increases.  Then $N_{r+1}^{k'}$ exists and fails to thicken as in the argument for Case I for $k':=k/q_{r+1}$, and $C_{r+1}$ will be the least such $k/q_{r+1} \in \mathbb{N}$.
\end{proof}

\section{Negative cablings that satisfy the UTP}
\label{sec:negativecablings}

We provide below the proof of Lemma 3.3, which is really just a matter of referencing a previous proof.

\begin{proof}
This is the case where $q_{r+1}/p_{r+1} \in (-1/A_r,0)$ in the $\mathcal{C}'$ framing, we know $K_r$ satisfies the inductive hypothesis, and we wish to show that $K_{r+1}$ satisfies the UTP.  The proof is identical to that of steps 1 and 2 in the proof of Theorem 1.5 from Section 6 of \cite{[L]}, the key being that since $-1/A_r < q_{r+1}/p_{r+1} < 0$, this cabling slope is shielded (in the Farey tessellation by an arc from $-1/A_r$ to $0$) from any $N_r^k$ that fail to thicken.
\end{proof}

\section{Transversely non-simple iterated torus knots}
\label{sec:nonsimple}

We have completed the UTP classification of iterated torus knots; it now remains to show that in the class of iterated torus knots, failing the UTP is a sufficient condition for supporting transversely non-simple cablings.  To this end, in this section we prove Theorem \ref{second theorem}; we do so by working through two lemmas.  The first lemma will give us information about just a piece of the Legendrian mountain range for $K_r = ((P_1,q_1),...,(P_r,q_r))$ where $P_i > 0$ for all $i$; in the second lemma, we will then use this information to obtain enough information about the Legendrian mountain ranges of certain cables $K_{r+1}$ to conclude that these cables are transversely non-simple.  We will therefore not be completing the Legendrian or transverse classification of these iterated torus knots.

\begin{lemma}
\label{sliceofmtnrange}
Suppose $K_r = ((P_1,q_1),...,(P_r,q_r))$ is an iterated torus knot where $P_i > 0$ for all $i$.  Then there exists Legendrian representatives $L_r^{\pm}$ with $\tb(L_r^{\pm}) = 0$ and $\rot(L_r^{\pm})=\pm(A_r-B_r)$; also, $L_r^{\pm}$ destabilizes.
\end{lemma}

\begin{proof}
The lemma is true for positive torus knots \cite{[EH2]}, so we inductively assume it is true for $K_{r-1}$.  Then look at Legendrian rulings $\widetilde{L}_r^{\pm}$ on standard neighborhoods of the inductive $L_{r-1}^{\pm}$.  In the $\mathcal{C}'$ framing the boundary slope of these $N(L_{r-1}^{\pm})$ is $-1/A_{r-1}$, and so a calculation shows that $t(\widetilde{L}_r^{\pm}) = -P_r$; hence $\tb(\widetilde{L}_r^{\pm})=A_r-P_r$.

To calculate the rotation number of $\widetilde{L}_r^{\pm}$, we use the following formula from \cite{[EH1]}, where $D$ is a convex meridian disc for $N(L_{r-1}^{\pm})$ and $\Sigma$ is a Seifert surface for the preferred longitude on $\partial N(L_{r-1}^{\pm})$:

\begin{eqnarray*}
\rot(\widetilde{L}_r^{\pm}) &=& P_r \rot(\partial D) + q_{r} \rot(\partial\Sigma)\\\nonumber
											 &=& \pm q_r (A_{r-1}-B_{r-1})\\\nonumber
											 &=& \pm (q_rA_{r-1} + p_r - q_rB_{r-1} -p_r)\\\nonumber
											 &=& \pm(P_r-B_r)
\end{eqnarray*}

This gives us

\begin{equation}
sl(T_{-}(\widetilde{L}_r^{+}))=(A_r-P_r)+(P_r-B_r)=A_r-B_r\nonumber
\end{equation}

and

\begin{equation}
sl(T_{+}(\widetilde{L}_r^{-}))=(A_r-P_r)-(-(P_r-B_r))=A_r-B_r\nonumber
\end{equation}

This, along with Lemma \ref{eulerchar}, shows us that $\widetilde{L}_r^{+}$ is on the right-most slope of the Legendrian mountain range of $K_r$, and $\widetilde{L}_r^{-}$ is on the left-most edge.  To the former we can perform positive stabilizations to reach $L_r^+$ at $\tb=0$ and $\rot=A_r-B_r$; to the latter we can perform negative stabilizations to reach $L_r^-$ at $\tb=0$ and $\rot=-(A_r-B_r)$ -- we know such stabilizations can be performed since $A_r-P_r > 0$. 
\end{proof}

So suppose $K_r$ is an iterated torus knot that fails the UTP (which is precisely when $P_i > 0$ for all $i$).  Then we know that for $k \geq C_r$ there exist non-thickenable solid tori $N_r^k$ having intersection boundary slopes of $-(k+1)/(A_rk+B_r)$, where these slopes are measured in the $\mathcal{C}'$ framing.  Switching to the standard $\mathcal{C}$ framing, these intersection boundary slopes are $(k+1)/(A_r-B_r) = -(k+1)/\chi(K_r)$.  Now as $k \rightarrow \infty$, there are infinitely many values of $k+1$ which are prime and greater than $A_r-B_r$.  As a consequence, there are infinitely many $N_r^k$ with two dividing curves.  Based on this observation, we make the following definition:

\begin{definition}
{\em Suppose $K_r = ((P_1,q_1),...,(P_r,q_r))$ is an iterated torus knot where $P_i > 0$ for all $i$. Let $\widehat{K}_{r+1}$ be a cabling of $K_r$ with $\mathcal{C}'$ slope $-(k+1)/(A_rk+B_r)$, where $-1/(A_r-1) < -(k+1)/(A_rk+B_r) < -1/A_r$ and there is an $N_r^k$ with two dividing curves that fails to thicken.}
\end{definition}

So given $K_r$, there are infinitely many such cabling knot types $\widehat{K}_{r+1}$, all of these being cablings of the form $(-\chi(K_r), k+1)$ as measured in the preferred framing.  The following lemma will then prove Theorem \ref{second theorem}.

\begin{lemma}
$\widehat{K}_{r+1}$ is a transversely non-simple knot type.
\end{lemma}

\begin{proof}
We first calculate $\chi(\widehat{K}_{r+1})$.  Using the recursive expression from Lemma \ref{eulerchar} we obtain

\begin{eqnarray*}
\chi(\widehat{K}_{r+1}) &=& q_{r+1} \chi(K_{r}) - P_{r+1} q_{r+1} + P_{r+1}\\\nonumber
												&=& (k+1)(-A_r+B_r) - (A_r-B_r)(k+1)+(A_r-B_r)\\\nonumber
												&=& (2k+1)(-A_r + B_r)\nonumber
\end{eqnarray*}

We now look at the two universally tight non-thickenable $N_r^k$ that have representatives of $\widehat{K}_{r+1}$ as Legendrian divides.  These Legendrian divides have $\tb=A_{r+1}=q_{r+1}P_{r+1}=(k+1)(A_r-B_r)$.  To calculate rotation numbers, we have two possibilities, depending on which boundary of the two universally tight $N_r^k$ the Legendrian divides reside.  Using the formula from \cite{[EH1]}, we obtain

\begin{eqnarray*}
\rot(\widehat{K}_{r+1}) &=& q_{r+1} \rot(\partial\Sigma) + P_{r+1} \rot(\partial D)\\\nonumber
										 &=& P_{r+1}(\pm(q_{r+1}-1))\\\nonumber
										 &=& \pm k (A_r-B_r)
\end{eqnarray*}

We will call the two Legendrian divides corresponding to $\rot=\pm k(A_r-B_r)$, $L_{r+1}^{\pm}$ respectively.  We can calculate the self-linking number for the negative transverse push-off of $L_{r+1}^+$ to be $sl = (2k+1)(A_r-B_r)=-\chi(\widehat{K}_{r+1})$.  This shows that $L_{r+1}^+$ is on the right-most edge of the Legendrian mountain range and is at $\overline{\tb}$.  Similarly, $L_{r+1}^-$ is on the left-most edge of the Legendrian mountain range and is at $\overline{\tb}$.

We now look at solid tori $\widehat{N}_r$ with intersection boundary slope $-(k+1)/(A_rk+B_r)$, but which {\em thicken} to solid tori with intersection boundary slopes $-1/(A_r-1)$.  Such tori $\partial \widehat{N}_r$ are embedded in universally tight basic slices bounded by tori with dividing curves of slope $-1/(A_r-1)$ and $-1/A_r$.  Legendrian divides on such $\widehat{N}_r$ have $\tb=(k+1)(A_r-B_r)$; to calculate possible rotation numbers for these Legendrian divides, we recall the procedure used in the proof of Theorem 1.5 in Section 6 of \cite{[L]}.  There, in equation 14, we used a formula for the rotation numbers from \cite{[EH1]}, where the range of rotation numbers was given by the following (substituting $A_r-1$ for $n$):

\begin{equation}
\rot(L_{r+1}) \in \left\{ \pm (p_{r+1} + (A_r -1)q_{r+1} + q_{r+1} \rot(L_r)) | \tb(L_r)=A_r-(A_r-1)=1 \right\}\nonumber
\end{equation}

Now from Lemma \ref{sliceofmtnrange} we know that there is an $L_r$ with $\tb(L_r)=1$ and $\rot(L_r)=-(A_r-B_r)+1$.  Plugging this value of the rotation number into the expression above yields $\rot(L_{r+1})=\pm k(A_r-B_r)$.  We will call the Legendrian divides having these rotation numbers $\widehat{L}_{r+1}^{\pm}$, respectively.  Important for our purposes is that $\widehat{L}_{r+1}^{\pm}$ have the same values of $\tb$ and $\rot$ as $L_{r+1}^{\pm}$.

We focus in, for the sake of argument, on $L_{r+1}^{-}$ and $\widehat{L}_{r+1}^{-}$, and we show that $T_-(L_{r+1}^{-})$ is not transversely isotopic to $T_-(\widehat{L}_{r+1}^{-})$, despite having the same self-linking number.

Consider first $T_+(L_{r+1}^{-})$.  It is in fact one of the dividing curves on $\partial N_r^k$, and is also at maximal self-linking number for $\widehat{K}_{r+1}$.  Similarly, $T_+(\widehat{L}_{r+1}^{-})$ is one of the dividing curves on $\partial \widehat{N}_r$, and is also at maximal self-linking number.  Now from \cite{[He1]} we know that $\widehat{K}_{r+1}$ is a fibered knot that supports the standard contact structure, since it is an iterated torus knot obtained by cabling positively at each iteration.  As a consequence, from \cite{[EV]}, we also know that $\widehat{K}_{r+1}$ has a unique transverse isotopy class at $\overline{sl}$.  Hence we know that $T_+(L_{r+1}^{-})$ and $T_+(\widehat{L}_{r+1}^{-})$ are transversely isotopic.  Thus there is a transverse isotopy (inducing an ambient contact isotopy) that takes these two dividing curves on the two different tori to each other.  Thus we may assume that $\partial N_r^k$ and $\partial \widehat{N}_r$ intersect along one component of the dividing curves; we call this component $\gamma_+$.

Now suppose, for contradiction, that $T_-(L_{r+1}^{-})$ is transversely isotopic to $T_-(\widehat{L}_{r+1}^{-})$.  These transverse knots are represented by the other two dividing curves on $\partial N_r^k$ and $\partial \widehat{N}_r$, respectively, and we are therefore assuming that there is a transverse isotopy taking one to the other.  This transverse isotopy will induce an ambient contact isotopy of $S^3$, including a contact isotopy of the two tori $\partial N_r^k$ and $\partial \widehat{N}_r$, with $\gamma_+$ sitting on both of them.  Since $\partial N_r^k$ and $\partial \widehat{N}_r$ are incompressible in $S^3 \backslash N(\gamma_+)$, we may assume that after a contact isotopy of $S^3$, $\partial N_r^k$ and $\partial \widehat{N}_r$ intersect along their two dividing curves, which we denote as $\gamma_+$ and $\gamma_-$.  We now observe that there is an isotopy (although, a priori, not necessarily a contact isotopy) of $\partial N_r^k$ to $\partial \widehat{N}_r$ relative to $\gamma_+$ and $\gamma_-$.  We claim that as a result $\widehat{N}_r$ cannot thicken, thus obtaining our contradiction.  We do this by noting that the isotopy of $\partial N_r^k$ to $\partial \widehat{N}_r$ relative to $\gamma_+$ and $\gamma_-$ may be accomplished by the attachment of successive bypasses, beginning with $\partial N_r^k$ and ending at $\partial \widehat{N}_r$; thus $\partial \widehat{N}_r$ is fixed throughout the process.  Since these bypasses are attached in the complement of the two dividing curves, none of these bypass attachments can change the boundary slope.  However, they may increase or decrease the number of dividing curves.  Starting with $T = \partial N_r^k$, we make the following inductive hypothesis, which we will prove is maintained after bypass attachments:   

\begin{itemize}
	\item[1.]  $T$ is a convex torus which contains $\gamma_+$ and $\gamma_-$, and thus has slope $-(k+1)/(A_rk+B_r)$.
	\item[2.]  $T$ is a boundary-parallel torus in a $[0,1]$-invariant $T^2 \times [0,1]$ with $\textrm{slope}(\Gamma_{T_0}) = \textrm{slope}(\Gamma_{T_1}) = -(k+1)/(A_rk+B_r)$, where the boundary tori have two dividing curves.
	\item[3.]  There is a contact diffeomorphism $\phi: S^3 \rightarrow S^3$ which takes $T^2 \times [0,1]$ to a standard $I$-invariant neighborhood of $\partial N_r^k$ and matches up their complements.
\end{itemize}

The argument that follows is similar to Lemma 6.8 in \cite{[EH1]}.  First note that item 1 is preserved after a bypass attachment, since such a bypass is in the complement of $\gamma_+$ and $\gamma_-$, and thus cannot change the slope of the dividing curves.  To see that items 2 and 3 are preserved, suppose that $T'$ is obtained from $T$ by a single bypass.  Since the slope was not changed, such a (non-trivial) bypass must either increase or decrease the number of dividing curves by 2.  Suppose first that the bypass is attached from the inside, so that $T' \subset N$, where $N$ is the solid torus bounded by $T$.  For convenience, suppose $T = T_{0.5}$ inside the $T^2 \times [0,1]$ satisfying items 2 and 3 of the inductive hypothesis.  Then we form the new $T^2 \times [0.5,1]$ by taking the old $T^2 \times [0.5,1]$ and adjoining the thickened torus between $T$ and $T'$.  Now $T'$ bounds a solid torus $N'$, and, by the classification of tight contact structures on solid tori, we can factor a nonrotative layer which is the new $T^2 \times [0,0.5]$.

Alternatively, if $T' \subset (S^3 \backslash N)$, then we know that $N'$ thickens to an $N_r^k$, and thus there exists a nonrotative outer layer $T^2 \times [0.5,1]$ for $S^3 \backslash N'$, where $T_1$ has two dividing curves.  Thus the proof is done, for after enough bypass attachments we will obtain $T=\partial \widehat{N}_r$, with $\widehat{N}_r$ non-thickenable.  But this is a contradiction, since $\widehat{N}_r$ does thicken. 
\end{proof}

\end{document}